\documentclass[12pt,draft]{article}

\usepackage{amsmath,amssymb,amsthm,xspace,amscd,tikz-cd}

\usetikzlibrary{babel}

\theoremstyle{remark}

\theoremstyle{plain}
\newtheorem{theorem}{Theorem}
\newtheorem*{theorem_nonumber}{Theorem}

\newtheorem{proposition}{Proposition}
\newtheorem*{proposition_nonumber}{Proposition}

\newcommand{\Id}{{\mathbb I\rm d}}

\newcommand{\be}{{\mathrm b}}
\newcommand{\en}{{\mathrm e}}

\newcommand{\A}{{\mathcal A}}
\newcommand{\Inv}{{\mathcal I}}
\newcommand{\R}{{\mathbb R}}
\newcommand{\Z}{{\mathbb Z}}

\newcommand{\vect}{{\mathbf v}}
\newcommand{\wect}{{\mathbf w}}
\newcommand{\uect}{{\mathbf u}}
\newcommand{\xect}{{\mathbf x}}
\newcommand{\yect}{{\mathbf y}}

\newcommand{\fect}{{\mathbf f}}
\newcommand{\Fect}{{\mathbf F}}

\newcommand{\Phict}{{\mathbf \Phi}}

\newcommand{\psl}{\mathrm{PSL}(2,\R)}

\begin{document}

\title{Linear-fractional rotational\\ interval exchange transformations}
\author{Alexey TEPLINSKY
\thanks{Institute of Mathematics NASU; Ukraine, Kyiv, Tereshchenkivska Street 3; e-mail:\ teplinsky.a@gmail.com}}
\maketitle
\hfill UDC 517.5
\begin{abstract}
Linear-fractional rotational interval exchange transformations appear 
naturally when investigating circle diffeomorphisms with breaks using 
the renormalization group approach. Earlier, we achieved an important 
rigidity result on circle diffeomorphisms with a single break by 
means of establishing a time-reversing duality for the corresponding 
commuting pairs of linear-fractional maps. In this paper, we extend 
the established duality to the case of multiple breaks, where a 
commuting pair from the previous case is replaced by a `commuting 
collection' of an even number of linear-fractional maps, which act according to a rotational interval exchange scheme of special type, 
while satisfying certain commutation relations.
\end{abstract}

\section{Introduction}
\label{sect:introduction}

Khanin and Vul~\cite{KhaninVul91} were the first to research dynamical properties of a circle diffeomorphism with a break (i.~e., an orientation-preserving circle homeomorphism that is smooth everywhere but a single point, at which there is a jump in its first derivative) by renormalization group analysis methods; they showed that the sequence of renormalizations for such a map approaches a space of certain linear-fractional functions. In the algorithm they used, the consecutive renormalizations are pairs of functions, which in the limit satisfy a commutation relation that includes the size of break (i.~e., the ratio of left and right derivatives at that singularity point) as a real parameter, and this is why the limiting space for consecutive renormalizations is called the space of linear-fractional commuting pairs. It is worth mentioning that similar notion of a commuting pair exists for renormalizations of holomorphic circle diffeomorphisms with critical singularity as well~\cite{Yampolsky03}. 

In our work~\cite{KhaninTeplinsky13}, the following result about circle diffeomorphisms with a break was proved: any two such $C^{2+\varepsilon}$-smooth diffeomorphisms, $\varepsilon>0$, with the same irrational rotation number (with certain restrictions) and the same size of break are not just topologically conjugate (which is implied by classic Denjoy theory), but also $C^1$-smooth conjugate. Among the principal tools in the proof of this result was an involutive operation of duality (which we invented in~\cite{Teplinsky08}) on the space of commuting pairs' parameters that reverses time for renormalization process, i.~e., this duality conjugates the renormalization operator acting on this space with the reverse renormalization operator. This time-reversing symmetry allowed us to find in the parameter space a two-dimentional structure known to dynamicists as a `Smail horseshoe' (or just as a hyperbolic horseshoe or a hyperbolic attractor), and afterwards the proof of main result was reduced to a relatively easy check on the conditions of so-called Conditional Theorem from~\cite{KhaninTeplinsky07}. 

Since then, we were steadily interested in extending the above-mentioned results onto the case, in which circle diffeomorphisms have not a single, but multiple breaks. In this paper, a substantial part of that program is established---namely, the notion of a commuting pair is generalized to a new notion of a commuting collection of $2N$ linear-fractional functions (where the integer $N\ge1$ corresponds to the number of breaks of a circle diffeomorphism under consideration); on these commuting collections, we define the induction (in the form of elementary induction steps of four types) and the duality operation that reverses time for induction process. Here it is worth to mention that processes of induction and renormalization are closely related: the action of a renormalization operator is a composition of a certain sequence of elementary induction steps (like in relation of Zorich `accelerated' induction~\cite{Zorich96} to Rauzy--Veech `elementary' induction~\cite{Veech78,Rauzy79} for classic interval exchange transformations~\cite{Keane75}) followed by a linear rescaling that spreads the newly decreased phase space of the dynamical system back to near-unit size. 

For defining and investigating commuting collections of linear-fractional functions and the involutive operation of duality, which acts on them and is crucial for our research, we use the formalism of interval rearrangement ensembles (IRE)~\cite{Teplinsky23} in the particular case of rotational interval exchange transformations~\cite{Teplinsky24} that actually correspond to rotations of a circle with marked points. In~\cite{Teplinsky23} we defined induction for IRE in the form of elementary induction steps of four types and duality as an operation on IRE schemes. Here we will expand both induction and duality to the space of commuting collections of algebraic linear-fractional functions, considered as elements of the projective special linear group $\psl$.

This paper's structure is the following: in Section~\ref{sect:basic} we remind the basic notions of our IRE theory developed in~\cite{Teplinsky23}, define nonlinear IREs and describe induction for them; in Section~\ref{sect:commuting} we define rotational schemes of special type and commuting collections of linear-fractional functions, and also describe induction for such collections; in Section~\ref{sect:duality} we present an algorithm allowing to calculate for a given commuting collection $\Fect$ the dual commuting collection $\Fect^*$ with the same sizes of breaks, for which purpose we define an array of node functions $\Phict$ and the dual array $\Phict^*$ positioned on a toroidal patchwork surface of natural extension; in Section~\ref{sect:symmetry} we prove the main result of this paper---theorem stating that the duality involution on commuting collections, which was defined in the previous section, reverses time for induction; in Section~\ref{sect:verification} we verify applicability of presented theory to circle diffeomorphisms with breaks by establishing certain connections to our previous and future works; the last Section~\ref{sect:summary} shortly summarizes our results.

This work was partially supported by Simons Foundation, grant SFI-PD-Ukraine-00014586.

\section{Basic concepts and definitions}
\label{sect:basic}

\subsection{Principal notions from the IRE theory}
\label{subsect:definitions}

Here we remind the main concepts of the IRE (interval rearrangement ensembles) theory developed by us in~\cite{Teplinsky23}. Let $\A$ be an alphabet of $d\ge1$ symbols (these will serve as labels for intervals that are being rearranged). By the definition~\cite{Teplinsky23}, an IRE \emph{scheme} is any permutation $\sigma$ of the doubled alphabet $\bar\A=\A\times\{\be,\en\}$ (the letters `$\be$' and `$\en$' stand for `beginning' and `ending' respectively).
An \emph{interval rearrangement ensemble} (IRE) is defined as a pair $(\sigma,\xect)$, where the scheme $\sigma$ is paired with a \emph{vector of endpoints} $\xect\in\R^{\bar\A}$, whose coordinates are \emph{allowed} by this scheme, i.~e.\ they satisfy the equalities
\begin{equation}\label{eq:endpoints_relation}
v_\alpha=x_{\sigma(\alpha\be)}-x_{\alpha\be}=x_{\alpha\en}-x_{\sigma(\alpha\en)}\quad{\text{for all}}\quad\alpha\in\A,
\end{equation} 
which also conveniently determine the \emph{vector of lengths} $\vect\in\R^{\A}$. If all the lengths $v_\alpha$ are positive, then the IRE $(\sigma,\xect)$ is called \emph{positive}. If for a scheme $\sigma$ there exists an allowed vector of lengths $\xect$ such that all the lengths (\ref{eq:endpoints_relation}) are positive, then the scheme $\sigma$ itself is called \emph{positive}.

An important point is that an IRE scheme $\sigma$, as any permutation, is decomposed in $N\ge1$ \emph{cycles} of the form $c=c(\bar\xi)=(\bar\xi,\sigma(\bar\xi),\dots,\sigma^k(\bar\xi))$, $\bar\xi\in\bar\A$, where $k\ge0$, $\sigma^{k+1}(\bar\xi)=\bar\xi$, $\sigma^i(\bar\xi)\ne\bar\xi$ for $0<i\le k$, and vectors $c(\bar\xi)$ and $c(\sigma^i(\bar\xi))$ are considered identical for all $i\in\Z$.

A positive IRE can be imagined as an ensemble of $2d$ intervals, which are coupled into $d$ pairs with labels $\alpha\in\A$; in every such a pair, the \emph{beginning interval} $I_{\alpha\be}=[x_{\alpha\be},x_{\sigma(\alpha\be)})$ and the \emph{ending interval} $I_{\alpha\en}=[x_{\sigma(\alpha\en)},x_{\alpha\en})$ with the same label $\alpha$ have the same length $v_\alpha$. Corresponding to the cycles in scheme $\sigma$, all beginning and ending intervals with respective subscripts are joined by their endpoints into $N$ circularly connected polygonal chains. Such a polygonal chain can be seen as a piecewise-linear one-dimensional circular path starting from $x_{\bar\xi}$ and going to $x_{\sigma(\bar\xi)}$ first, then continuing from $x_{\sigma(\bar\xi)}$ to $x_{\sigma^2(\bar\xi)}$, and so on until returning to $x_{\bar\xi}$; on this path, every beginning interval is passed from left to right, while every ending interval is passed from right to left. 

In the view of this IRE concept, we define a \emph{interval exchange scheme} $\sigma$ as a positive IRE scheme such that every cycle in it can be broken into two arcs in such a way that the first arc contains only beginning intervals, and the second one contains only ending intervals---i.~e., every cicle has the form $c=(\alpha_1\be,\dots,\alpha_m\be,\beta_n\en,\dots,\beta_1\en)$ for certain labels $\alpha_1,\dots,\alpha_m$, $\beta_1,\dots,\beta_n\in\A$, $n,m\ge1$. Cycles of such a form $(\alpha_1\be,\dots,\alpha_m\be,\beta_n\en,\dots,\beta_1\en)$ can be conveniently written down in the `two-row notation' format as
$\left[\begin{array}{ccc}
\alpha_1 & \dots & \alpha_m \\
\beta_1 & \dots & \beta_n                                                                                                                                           \end{array}\right]$. 

Accordingly, a positive IRE $(\sigma,\xect)$ is called an \emph{interval exchange transformation} (on multiple segments), if its scheme $\sigma$ is an interval exchange scheme. We associate this object with the discrete-time dynamical system, the phase space of which is the disjoint union $\bigsqcup_{s=1}^NJ_s$ of half-open segments $J_1,\dots,J_N\subset\R$ corresponding to the cycles $c_1,\dots,c_N$ in the scheme $\sigma$: namely, a segment $J=[x_{\alpha_1\be},x_{\beta_n\en})=\bigcup_{i=1}^mI_{\alpha_i\be}=\bigcup_{i=1}^nI_{\beta_i\en}$ corresponds to the cycle $c=\left[\begin{array}{ccc}
\alpha_1 & \dots & \alpha_m \\
\beta_1 & \dots & \beta_n                                                                                                                                           \end{array}\right]$. 
A dynamical system of interval exchange transformation is induced by the discontinuous mapping on $\bigsqcup_{s=1}^NJ_s$ that shifts every beginning interval $I_{\alpha\be}$ onto the corresponding ending interval $I_{\alpha\en}$, $\alpha\in\A$. (Here and further, we name $I_{\bar\xi}$, $\bar\xi\in\bar\A$, by the word `intervals', while $J_i$, $1\le i\le N$, we name by the word `segments', in order to distinguish between these two groups of essentially different objects, although formally both being half-open intervals.)

Our definition of interval exchange transformation generalizes the classical one only in that intervals are being exchanged not over a single, but over multiple segments, and the left endpoints of those segments are not necessarily located at the zero coordinate. However, the approach to interval exchanges through our IRE theory allows to use its key concept, which is the duality that reverses time for induction of Rauzy--Veech~\cite{Veech78,Rauzy79} type on schemes.

Two IRE schemes $\sigma$ and $\sigma^*$ are called \emph{dual} to each other, if
\begin{equation}\label{eq:duality}
\sigma^*(\alpha\be)=\sigma(\alpha\en),\quad
\sigma^*(\alpha\en)=\sigma(\alpha\be)\qquad\text{for all}\qquad\alpha\in\A.
\end{equation}

Obviously, $\sigma^{**}=\sigma$, i.~e.\ the \emph{duality operation} $\Inv$ on schemes defined as $\sigma^*=\Inv\sigma$, according to the formulas above, is an involution: $\Inv^2=\Id$.

Through this duality, we define rotational interval exchange transformations. An interval exchange scheme $\sigma$ is called \emph{rotational}, if the dual IRE scheme $\sigma^*$ is an interval exchange scheme as well. (It is clear that, by this definition, if the scheme $\sigma$ is rotational, then the scheme $\sigma^*$ is rotational as well, and vice versa.) Respectively to this, an inteval exchange transformation $(\sigma,\xect)$ is called \emph{rotational}, if its scheme $\sigma$ is rotational. 

Our recent paper~\cite{Teplinsky24} addresses rotational interval exchange transformations and describes exactly their relation to circle rotations. Not going into detail, one can say that a rotational interval exchange transformation, as a dynamical system, is effectively a first return map for a circle rotation onto a union of finite number of circle arcs.
(A unit circle rotation is formally defined as an automorphism $x\mapsto x+\rho$ acting on the factor space $\R/\Z$; it is called irrational if such is its rotation number $\rho$.)

\subsection{Nonlinear and linear-fractional IREs}
\label{subsect:nonlinear}

A classical interval exchange transformation~\cite{Keane75} is a dynamical system on a segment split into a finite number of intervals, everyone of whom is being rigidly shifted along the axis, and the union of their images is the same segment. The map inducing this dynamical system is one-to-one; on every beginning interval $I_{\alpha\be}$ it has the form of a \emph{shift} $f_\alpha:x\mapsto x+\ell_\alpha$ by certain constant distance $\ell_\alpha\in\R$, and the image of this beginning interval is the ending interval $I_{\alpha\en}=f_\alpha(I_{\alpha\be})$, $\alpha\in\A$. In IRE formalism, the beginning interval $I_{\alpha\be}=[x_{\alpha\be},x_{\sigma(\alpha\be)})$ is shifted onto the ending interval $I_{\alpha\en}=[x_{\sigma(\alpha\en)},x_{\alpha\en})$ according to the scheme $\sigma$, hence $\ell_\alpha=x_{\sigma(\alpha\en)}-x_{\alpha\be}=x_{\alpha\en}-x_{\sigma(\alpha\be)}$, $\alpha\in\A$.

A natural generalization of such a dynamical system is replacing shifts with arbitrary continuous increasing functions
$f_\alpha$, which results in \emph{nonlinear interval exchange transformations} (they are also known as GIET, i.~e.\ \emph{generalized interval exchange transformations}, but we find the latter naming a really poor choice, since the term `generalized' is too broad and therefore meaningless). In such a case, length of the ending interval $I_{\alpha\en}=f_\alpha(I_{\alpha\be})$ does not need to be the same as length of the beginning interval $I_{\alpha\be}$, however the union of all the ending intervals must coincide with the union of all the beginning intervals so that the united map is one-to-one. 

Formally, what we call \emph{nonlinear positive IRE} is a set $(\sigma,\xect,\fect)$ consisting of an IRE scheme $\sigma$, a vector of endpoints $\xect\in\R^\A$, whose coordinates satisfy
\begin{equation}\label{eq:endpoints_relation_nonlinear}
x_{\sigma(\alpha\be)}-x_{\alpha\be}>0, \quad x_{\alpha\en}-x_{\sigma(\alpha\en)}>0\quad{\text{for all}}\quad\alpha\in\A,
\end{equation} 
and also a set $\fect$ of continuous strictly increasing functions $f_\alpha$ that map the beginning intervals $I_{\alpha\be}=[x_{\alpha\be},x_{\sigma(\alpha\be)})$ onto the ending intervals $I_{\alpha\en}=[x_{\sigma(\alpha\en)},x_{\alpha\en})$, $\alpha\in\A$. 

If $\sigma$ is an interval exchange scheme, then a nonlinear positive IRE $(\sigma,\xect,\fect)$ is called a \emph{nonlinear interval exchange transformation} and associated with the respective dynamical system on the disjoint union $\bigsqcup_{s=1}^NJ_s$ of segments $J_1,\dots,J_N\subset\R$ that correspond to the cycles $c_1,\dots,c_N$ in the scheme $\sigma$ (see the previous subsection). This dynamical system is induced by the piecewise continuous mapping $f:x\mapsto f_\alpha(x)$ for $x\in I_{\alpha\be}$, $\alpha\in\A$.  In the case of a rotational scheme $\sigma$, the set $(\sigma,\xect,\fect)$ and the respective dynamical system are both called a \emph{nonlinear rotational interval exchange transformation}.

It is clear that if all $f_\alpha$ are shifts, then the equalities (\ref{eq:endpoints_relation}) are satisfied, and we have an interval exchange transformation from the previous subsection, therefore they are a particular case of nonlinear interval exchange transformations. Other important particular cases of nonlinear interval exchange transformations are \emph{affine} and \emph{linear-fractional} ones: by these we mean the cases when all the functions in the set  $\fect$ are affine or, respectively, linear-fractional.


\subsection{Induction for nonlinear positive IRE}
\label{subsect:induction}

The main tool we apply to IREs is induction of Rauzy--Veech kind in the form of four elementary induction steps defined in Sect.~7 of our work~\cite{Teplinsky23}. Their action is being naturally generalized onto the case of nonlinear positive IREs defined in the previous subsection, and in this subsection we will describe it, but first let us remind how their action is defined on schemes.

We consider \emph{elementary induction steps} of four types: `cropping a beginning interval on the right', `cropping an ending interval on the right', `cropping a beginning interval on the left', and `cropping the ending interval on the left', denoting these operations as $\Pi_{\alpha\beta}^{\mathrm{rb}},$ $\Pi_{\alpha\beta}^{\mathrm{re}},$ $\Pi_{\alpha\beta}^{\mathrm{lb}}$, and $\Pi_{\alpha\beta}^{\mathrm{le}}$ respectively. (The superscripts `r' and `l' here come from the words `right' and `left', while  $\alpha,\beta\in\A$, $\alpha\ne\beta$, are the labels of the beginning and ending intervals, respectively, engaged in the operation.)
 
The action of these four steps on IRE schemes is easy to describe in terms of cycles in the permutation $\sigma$, namely: the step $\Pi_{\alpha\beta}^{\mathrm{rb}}$ (the steps $\Pi_{\alpha\beta}^{\mathrm{re}},$ $\Pi_{\alpha\beta}^{\mathrm{lb}},$ $\Pi_{\alpha\beta}^{\mathrm{le}}$) move the element $\beta\en$ (the elements $\alpha\be,$ $\beta\en,$ $\alpha\be$) from its current position into a position straight before $\alpha\en$ (straight after $\beta\be,$ straight after $\alpha\en,$ straight before $\beta\be$). (By saying that $\bar\xi$ stands straight before $\bar\eta,$ and $\bar\eta$ stands straight after $\bar\xi,$ we mean that $\sigma(\bar\xi) = \bar\eta,$ $\bar\xi, \bar\eta\in\bar\A.$)

Formal definitions below explicitly describe the schemes $\sigma'=\Pi\sigma$ obtained by applying each of the steps $\Pi_{\alpha\beta}^{\mathrm{rb}},$ $\Pi_{\alpha\beta}^{\mathrm{re}},$ $\Pi_{\alpha\beta}^{\mathrm{lb}}$, and $\Pi_{\alpha\beta}^{\mathrm{le}}$ in place of $\Pi$.
\begin{description}
\item[$\Pi_{\alpha\beta}^{\mathrm{rb}}$:]
if $\sigma(\beta\en) = \alpha\en,$ then $\sigma' = \sigma;$
otherwise $\sigma'(\alpha\be) = \sigma(\beta\en),$ $\sigma'(\beta\en) = \alpha\en,$ $\sigma'(\sigma^{-1}(\alpha\en)) = \beta\en,$
while $\sigma'(\bar\xi) = \sigma(\bar\xi)$ for all $\bar\xi\in\bar\A\,\backslash\{\alpha\be, \beta\en, \sigma^{-1}(\alpha\en)\}.$

\item[$\Pi_{\alpha\beta}^{\mathrm{re}}$:]
if $\sigma(\beta\be) = \alpha\be,$ then $\sigma' = \sigma;$
otherwise $\sigma'(\sigma^{-1}(\alpha\be)) = \beta\en,$ $\sigma'(\beta\be) = \alpha\be,$ $\sigma'(\alpha\be)  = \sigma(\beta\be),$
while $\sigma'(\bar\xi) = \sigma(\bar\xi)$ for all $\bar\xi\in\bar\A\,\backslash\{\sigma^{-1}(\alpha\be), \beta\be, \alpha\be\}.$

\item[$\Pi_{\alpha\beta}^{\mathrm{lb}}$:]
if $\sigma(\alpha\en) = \beta\en,$ then $\sigma' = \sigma;$
otherwise $\sigma'(\sigma^{-1}(\beta\en)) = \alpha\be,$ $\sigma'(\alpha\en) = \beta\en,$ $\sigma'(\beta\en) = \sigma(\alpha\en),$
while $\sigma'(\bar\xi) = \sigma(\bar\xi)$ for all $\bar\xi\in\bar\A\,\backslash\{\sigma^{-1}(\beta\en), \alpha\en, \beta\en\}.$

\item[$\Pi_{\alpha\beta}^{\mathrm{le}}$:]
if $\sigma(\alpha\be) = \beta\be,$ then $\sigma' = \sigma;$
otherwise $\sigma'(\beta\en) = \sigma(\alpha\be),$ $\sigma'(\alpha\be) = \beta\be,$ $\sigma'(\sigma^{-1}(\beta\be)) = \alpha\be,$ 
while $\sigma'(\bar\xi) = \sigma(\bar\xi)$ for all $\bar\xi\in\bar\A\,\backslash\{\beta\en, \alpha\be, \sigma^{-1}(\beta\be)\}.$
\end{description}

Proposition~3 from the work~\cite{Teplinsky23} completely describes both the domain and the range for every of these steps on schemes, let us cite it here: 

\begin{proposition_nonumber}[\cite{Teplinsky23}]
The elementary induction step $\Pi_{\alpha\beta}^{\mathrm{rb}}$ (the steps $\Pi_{\alpha\beta}^{\mathrm{re}},$ $\Pi_{\alpha\beta}^{\mathrm{lb}},$ $\Pi_{\alpha\beta}^{\mathrm{le}}$) acts on IRE schemes as a bijection from the set of all $\sigma$ such that $\sigma(\alpha\be) = \beta\en$ (such that $\sigma(\alpha\be) = \beta\en,$ $\sigma(\beta\en) = \alpha\be,$ $\sigma(\beta\en) = \alpha\be${)} onto the set of all $\sigma'$ such that $\sigma'(\beta\en) = \alpha\en$ (such that $\sigma'(\beta\be) = \alpha\be,$ $\sigma'(\alpha\en) = \beta\en,$ $\sigma'(\alpha\be) = \beta\be${).}
\end{proposition_nonumber}

Now, let us be given a nonlinear positive IRE $(\sigma,\xect,\fect)$. We say that an elementary induction step $\Pi_{\alpha\beta}^{\mathrm{rb}}$ (steps $\Pi_{\alpha\beta}^{\mathrm{re}}$, $\Pi_{\alpha\beta}^{\mathrm{lb}}$, $\Pi_{\alpha\beta}^{\mathrm{le}}$) is \emph{applicable} to this IRE, if the equality $\sigma(\alpha\be)=\beta\en$ (the equalities $\sigma(\alpha\be)=\beta\en$, $\sigma(\beta\en)=\alpha\be$, $\sigma(\beta\en)=\alpha\be$) and the inequality $x_{\alpha\be}<x_{\sigma(\beta\en)}$ (the inequalities $x_{\alpha\be}>x_{\sigma(\beta\en)}$, $x_{\beta\en}<x_{\sigma(\alpha\be)}$, $x_{\beta\en}>x_{\sigma(\alpha\be)}$) are satisfied, $\alpha,\beta\in\A$. Every equality here guaranties that the corresponding step is defined on the scheme $\sigma$, while every inequality guarantees that the nonlinear IRE $(\sigma',\xect',\fect')$ obtained from $(\sigma,\xect,\fect)$ by application of that step, stays positive. The action of steps on a scheme $\sigma$ we defined above, and now we define how the sets of endpoints and functions are changed. 

In each of the four cases, coordinates of exactly two endpoints get changed, as shown by the formulas below.

\begin{description}
\item[$\Pi_{\alpha\beta}^{\mathrm{rb}}$:]
$x'_{\beta\en}=x_{\alpha\en}$, $x'_{\alpha\en}=f_\alpha(x_{\sigma(\beta\en)})$, while 
$x'_{\bar\xi}=x_{\bar\xi}$ for all $\bar\xi\in\bar\A\backslash\{\beta\en,\alpha\en\}$.

\item[$\Pi_{\alpha\beta}^{\mathrm{re}}$:]
$x'_{\beta\en}=x_{\alpha\be}$, $x'_{\alpha\be}=f^{-1}_\beta(x_{\alpha\be})$, while 
$x'_{\bar\xi}=x_{\bar\xi}$ for all $\bar\xi\in\bar\A\backslash\{\beta\en,\alpha\be\}$.

\item[$\Pi_{\alpha\beta}^{\mathrm{lb}}$:]
$x'_{\alpha\be}=x_{\beta\en}$, $x'_{\beta\en}=f_\alpha(x_{\beta\en})$, 
while 
$x'_{\bar\xi}=x_{\bar\xi}$ for all $\bar\xi\in\bar\A\backslash\{\alpha\be,\beta\en\}$.

\item[$\Pi_{\alpha\beta}^{\mathrm{le}}$:]
$x'_{\alpha\be}=x_{\beta\be}$, $x'_{\beta\be}=f^{-1}_\beta(x_{\sigma(\alpha\be)})$, 
while 
$x'_{\bar\xi}=x_{\bar\xi}$ for all $\bar\xi\in\bar\A\backslash\{\alpha\be,\beta\be\}$.
\end{description}

In each of the four cases, exactly one function gets changed, as shown by the formulas below.

\begin{description}
\item[$\Pi_{\alpha\beta}^{\mathrm{rb}}$:]
$f'_\beta=f_\alpha\circ f_\beta$, while $f'_\xi=f_\xi$ for all $\xi\in\A\backslash\{\beta\}$.

\item[$\Pi_{\alpha\beta}^{\mathrm{re}}$:]
$f'_\alpha=f_\alpha\circ f_\beta$, while $f'_\xi=f_\xi$ for all $\xi\in\A\backslash\{\alpha\}$.

\item[$\Pi_{\alpha\beta}^{\mathrm{lb}}$:]
$f'_\beta=f_\alpha\circ f_\beta$, while $f'_\xi=f_\xi$ for all $\xi\in\A\backslash\{\beta\}$.

\item[$\Pi_{\alpha\beta}^{\mathrm{le}}$:]
$f'_\alpha=f_\alpha\circ f_\beta$, while $f'_\xi=f_\xi$ for all $\xi\in\A\backslash\{\alpha\}$.
\end{description}
(Strictly saying, in each case the second function among $f_\alpha$ and $f_\beta$ also gets changed, namely its domain gets cropped in accordance with the endpoints change---but the function does not change over that cropped interval.  It is worth adding here that in the main part of this paper we will be dealing with algebraic linear-fractional functions, which are not connected to intervals.)

The reason we gave this namely definition of induction steps for studying dynamical systems is demonstrated by the next Proposition regarding the case of a nonlinear interval exchange transformation (which is, on the one hand, a one-to-one map of a disjoint union of segments onto itself, and, on the other, a particular case of nonlinear positive IREs).

\begin{proposition}
Any elementary induction step that is applicable to a given nonlinear interval exchange transformation transforms it into a nonlinear interval exchange transformation, which is the first return map onto the appropriately cropped disjoint union of segments.
\end{proposition}

\begin{proof}
In each of the four cases, this statement is checked directly through formulas above. Consider the case of cropping a beginning interval on the right; the other cases are done similarly.

So, let an elementary induction step $\Pi_{\alpha\beta}^{\mathrm{rb}}$, $\alpha,\beta\in\A$, be applicable to a nonlinear interval exchange transformation $(\sigma,\xect,\fect)$. By definition of applicability this means that one of the cycles in the permutation $\sigma$ has the form $(\alpha_1\be,\dots,\alpha_m\be,\alpha\be,\beta\en,\beta_n\en,\dots,\beta_1\en)$, or $\left[\begin{array}{cccc}
\alpha_1 & \dots & \alpha_m & \alpha \\
\beta_1 & \dots & \beta_n & \beta                                                                                                                                           \end{array}\right]$ in two-row notation, for certain $\alpha_1,\dots,\alpha_m,\beta_1,\dots,\beta_n\in\A$, and that within the corresponding segment $J=[x_{\alpha_1\be},x_{\beta\en})=I_{\alpha_1\be}\cup\dots\cup I_{\alpha_m\be}\cup I_{\alpha\be}=I_{\beta_1\en}\cup\dots\cup I_{\beta_n\en}\cup I_{\beta\en}$ the rightmost beginning interval $I_{\alpha\be}=[x_{\alpha\be},x_{\beta\en})$ is longer than the rightmost ending interval $I_{\beta\en}=[x_{\beta_n\en},x_{\beta\en})$. Cropping $I_{\beta\en}$ away from $J$ creates a new, decreased in length, segment $J'=J\backslash[x_{\beta_n\en},x_{\beta\en})$, while the other segments do not change. It is easy to see that the nonlinear interval exchange transformation $(\sigma',\xect',\fect')$ determined by the above-written formulas for $\Pi_{\alpha\beta}^{\mathrm{rb}}$ is indeed the first return map for the starting interval exchange transformation onto the newly cropped phase space (the disjoint union of segments). Points lying on the beginning intervals $I_{\xi\be}$, $\xi\ne\beta$, get returned to the cropped phase space in one step, therefore the functions $f_\xi$, $\xi\ne\beta$, stay unchanged. And points of the beginning interval $I_{\beta\be}$ return in two steps: first, the function $f_\beta$ moves them to the interval $I_{\beta\en}$, which was cropped away, and then, as the second step, the function $f_\alpha$ returns them to the decreased phase space---therefore in the newly obtained system the function $f_\alpha\circ f_\beta$ acts on the interval $I_{\beta\be}$.

Proposition is proved.
\end{proof}

\section{Commuting collections of linear-fractional functions}
\label{sect:commuting}

\subsection{Rotational schemes of special type and commuting collections}
\label{subsect:commuting_collections}

Consider an alphabet $\A$ consisting of $d=2N$ symbols, $N\ge1$, wtih $N$ special symbols $\omega_1,\dots,\omega_N$ among them. The corresponding beginning elements $\omega_1\be$, \dots, $\omega_N\be$ of the doubled alphabet $\bar\A=\A\times\{\be,\en\}$ (which consists of $4N$ elements) we call \emph{stationary}. A rotational interval exchange scheme $\sigma$, which is a permutation on $\bar\A$, is called a \emph{rotational scheme of special type} (with a given set of stationary elements), if this scheme, together with the dual scheme $\sigma^*$, consist of exactly $N$ cycles each, and every cycle contains exactly one stationary element. Formally: $\sigma=\{c_1,\dots,c_N\}$, $\omega_i\be\in c_i$, $1\le i\le N$; $\sigma^*=\{z_1,\dots,z_N\}$, $\omega_i\be\in z_i$, $1\le i\le N$. In two-row notation:
\begin{gather*}
c_i=\left[\begin{array}{ccc}
\alpha_{1i} & \dots & \alpha_{m_i i} \\
\beta_{1i} & \dots & \beta_{n_i i}                                                                                                                                           \end{array}\right],\quad \omega_i\in\{\alpha_{1i}, \dots, \alpha_{m_i i}\},\quad 1\le i\le N;\\ 
z_i=\left[\begin{array}{ccc}
\gamma_{1i} & \dots & \gamma_{s_i i} \\
\delta_{1i} & \dots & \delta_{t_i i}                                                                                                                                           \end{array}\right],\quad \omega_i\in\{\gamma_{1i}, \dots, \gamma_{s_i i}\},\quad 1\le i\le N.
\end{gather*}
It is clear that the scheme, which is dual to a rotational scheme of special type, is a rotational scheme of special type as well, with the same set of stationary elements.

Let us be given a set of positive numbers $\nu_1,\dots,\nu_N$. They correspond to the sizes of breaks for a circle diffeomorphism with breaks, which renormalizations are studied, and we call these numbers accordingly as \emph{sizes of breaks}. 

Consider a nonlinear interval exchange transformation  $(\sigma,\xect,\Fect)$ with a rotational scheme of special type $\sigma$, the coordinates of \emph{stationary endpoints} $x_{\omega_i\be}=0$, $1\le i\le N$ (they correspond to the points of breaks for a circle diffeomorphism with breaks), and a set of linear-fractional maps $\Fect=(F_\alpha)_{\alpha\in\A}$, which satisfy (as algebraic functions) the following formal commutation relations along the cycles of the dual scheme $\sigma^*$:
\begin{equation}\label{eq:commutation}
\tilde F_{\gamma_{s_i i}}\circ\dots\circ \tilde F_{\gamma_{1i}}=
 F_{\delta_{t_i i}}\circ\dots\circ F_{\delta_{1i}},\quad 1\le i\le N,
\end{equation}
where the \emph{adjusted functions} $\tilde F_\alpha$, $\alpha\in\A$, are defined as $\tilde F_{\omega_i}=F_{\omega_i}\circ \nu_i\Id$, $1\le i\le N$, and $\tilde F_\alpha=F_\alpha$ for $\alpha\not\in\{\omega_1,\dots,\omega_N\}$, while $\nu_i\Id$ denotes a linear function $x\mapsto\nu_ix$. (Effectively, each one of $N$ commutation relations (\ref{eq:commutation}) contains exactly one size of break $\nu_i$, and its position in the left-hand side of the $i$th equality is determined by the position of a corresponding stationary element $\omega_i\be$ in the cycle $z_i$ in the dual scheme $\sigma^*$.)

A set of $2N$ algebraic increasing linear-fractional functions $\Fect=(F_\alpha)_{\alpha\in\A}$, which satisfy given commutation relations (\ref{eq:commutation}), is called a \emph{commuting collection}.

Here and further on, by saying `algebraic increasing linear-fractional functions' instead of `maps', we emphasize their disconnection from domains of definition. We identify such functions with elements of the projective special linear group $\psl$, which we define as the result of factorizing the group of all real matrices $2\times2$ with positive determinants $\mathrm{GL^+}(2,\R)=\left\{\left(\begin{array}{cc}
a & b\\
c & d                                                                                                                                           \end{array}\right): a,b,c,d\in\R, ad-bc>0\right\}$ modulo its center---that is the subgroup of all scalar matrices $\mathrm{Z}(2,\R)=\left\{\left(\begin{array}{cc}
a & 0\\
0 & a                                                                                                                                           \end{array}\right):a\ne0\right\}$. Thus $\psl=\mathrm{GL^+}(2,\R)/\mathrm{Z}(2,\R)$, and we naturally identify an algebraic increasing linear-fractional function $F:x\mapsto\displaystyle\frac{ax+b}{cx+d}$ with the element $\left(\begin{array}{cc}
a & b\\
c & d                                                                                                                                           \end{array}\right)$ of this group (strictly saying---with the equivalence class $\left(\begin{array}{cc}
a & b\\
c & d                                                                                                                                           \end{array}\right)\mathrm{Z}(2,\R)\in\psl$, but we will not be pointing at this later in the text for simplicity). The composition operation for such functions corresponds exactly to the group operation on $\psl$, which is effectively the matrix product, and therefore the commutation relations (\ref{eq:commutation}) can be seen as matrix products. Let us specially mention here that the factors $\nu_i\Id$ in these relations should not be confused with scalar matrices: as we said above, these factors are linear functions, with corresponding matrices $\left(\begin{array}{cc}
\nu_i & 0\\
0 & 1                                                                                                                                           \end{array}\right)$.

We have to remark that, when studying circle diffeomorphisms with breaks, it is natural to assume that all sizes of breaks are different from~1, since unit size of break effectively means the absence of break; also, a unit size of break can change dimension of the space of commuting collections: for ex., in the case $\nu\ne1$ the commutation relation $F\circ G\circ\nu\Id=G\circ F$ determines a 3-parameter space of commuting collections $F,G\in\psl$, while in the case $\nu=1$ the same relation determines a 4-parameter space. However, for the theory we develop here, such a restriction is unnecessary, and so we do not impose it. Let us also mention that in the case when the product of all sizes of breaks is unit, i.~e.\ $\nu_1\nu_2\dots\nu_N=1$, the renormalizations of a corresponding circle diffeomorphism with breaks under broad assumptions approach a certain space of not just linear-fractional, but affine functions, as was shown, in particular, in~\cite{CunhaSmania13}.

Definitely, not every commuting collection of $\psl$ elements determines a real linear-fractional interval exchange transformation, because one cannot guarantee that the corresponding linear-fractional functions are continuous on the corresponding intervals. But if it is assumed that the linear-fractional interval exchange transformation $(\sigma,\xect,\Fect)$ described above really exists, then the formal commutation relations (\ref{eq:commutation}) contain the whole information on alphabet, on stationary elements and sizes of breaks (in the case of a unit size of break, the corresponding stationary element must be specially marked), on scheme $\sigma^*$ and therefore on scheme $\sigma$ as well. Moreover, if the commutation relations (\ref{eq:commutation}) are satisfied by a given commuting collection $\Fect$, then the vector $\xect$ is determined: $N$ stationary endpoints are fixed at the origin, and the rest $3N$ endpoints are determined as consecutive images of these $N$ stationary endpoints under action of corresponding maps from the given collection (along the cycles in the scheme $\sigma^*$). Therefore we see that all the information about the nonlinear interval exchange transformation $(\sigma,\xect,\Fect)$ is contained in the commutation relations (\ref{eq:commutation}) and a commuting (according to them) collection of algebraic linear-fractional functions $F_\alpha\in\psl$, $\alpha\in\A$.

\subsection{Induction for commuting collections and allowed steps}
\label{subsect:prohibited}

Applying induction to rotational schemes of special type, we want to guarantee that they stay rotational schemes of special type; we also want to guarantee that the coordinates of stationary endpoints $x_{\omega_i\be}$, $1\le i\le N$, in corresponding nonlinear interval exchange transformations $(\sigma,\xect,\Fect)$ described in Subsection~\ref{subsect:commuting_collections} do not change. Accordingly to formulas for endpoints change under the action of induction steps (written down in Subsection~\ref{subsect:induction}), we must exclude from our consideration the steps $\Pi_{\omega_i\beta}^{\mathrm{re}}$, $\Pi_{\omega_i\beta}^{\mathrm{lb}}$, $\Pi_{\alpha\omega_i}^{\mathrm{le}}$, and $\Pi_{\omega_i\beta}^{\mathrm{le}}$ for all $1\le i\le N$, $\alpha,\beta\in\A$, and so we call these steps \emph{prohibited} (for a given set of stationary elements $\omega_i\be$, $1\le i\le N$). We call induction steps being \emph{ruining positiveness} (for a given rotational scheme of special type $\sigma$), if their application to that scheme makes it not positive (and therefore not rotational). Elementary induction steps, which are defined (see the cited Proposition from Subsection~\ref{subsect:induction}) for a given rotational scheme of special type, and are neither prohibited, nor ruining positiveness, we call \emph{allowed}. It is easy to see that the result of applying any allowed elementary induction step to a rotational scheme of special type $\sigma$ is a certain rotational scheme of special type $\sigma'$ with the same stationary elements $\omega_i\be$, $1\le i\le N$. 

Elementary induction steps can be applied directly to a given commuting collection $\Fect=(F_\alpha)_{\alpha\in\A}$ accordingly to the formulas for nonlinear IREs written down in Subsection~\ref{subsect:induction}; their action results in a new set of linear-fractional functions $\Fect'=(F'_\alpha)_{\alpha\in\A}$, which we present below for the four cases $\Fect'=\Pi\Fect$, where $\Pi$ denotes one of the steps $\Pi_{\alpha\beta}^{\mathrm{rb}}$, $\Pi_{\alpha\beta}^{\mathrm{re}}$, $\Pi_{\alpha\beta}^{\mathrm{lb}}$, $\Pi_{\alpha\beta}^{\mathrm{le}}$, $\alpha,\beta\in\A$.

\begin{description}
\item[$\Pi_{\alpha\beta}^{\mathrm{rb}}$:]
$F'_\beta=F_\alpha\circ F_\beta$; while $F'_\xi=F_\xi$ for $\xi\ne\beta$.

\item[$\Pi_{\alpha\beta}^{\mathrm{re}}$:]
$F'_\alpha=F_\alpha\circ F_\beta$; while $F'_\xi=F_\xi$ for $\xi\ne\alpha$.

\item[$\Pi_{\alpha\beta}^{\mathrm{lb}}$:]
$F'_\beta=F_\alpha\circ F_\beta$; while $F'_\xi=F_\xi$ for $\xi\ne\beta$.

\item[$\Pi_{\alpha\beta}^{\mathrm{le}}$:]
$F'_\alpha=F_\alpha\circ F_\beta$; while $F'_\xi=F_\xi$ for $\xi\ne\alpha$.
\end{description}

\begin{proposition}\label{prop:induction_cc}
Any allowed elementary induction step transforms a given commuting collection into a commuting collection with the same stationary elements and the same sizes of breaks. 
\end{proposition}

\begin{proof}
Consider the case of a step $\Pi_{\alpha\beta}^{\mathrm{rb}}$. The condition $\sigma(\alpha\be)=\beta\en$ of that step being defined on the scheme $\sigma$ together with the definition of duality imply the equality $\sigma^*(\alpha\en)=\beta\en$ for the dual scheme $\sigma^*$. Therefore one of the cycles in it has the form
$\left[\begin{array}{ccc}
\dots & \dots & \dots \\
\dots & \beta\quad \alpha & \dots                                                                                                                                           \end{array}\right]$ (in two-row notation), and accordingly to this the product $F_\alpha\circ F_\beta$ is included in the right-hand side of one of the commutation relations (\ref{eq:commutation}). Next, the action of the step $\Pi_{\alpha\beta}^{\mathrm{rb}}$ on the scheme $\sigma$ moves the element $\beta\en$ from its current position into the position straight before $\alpha\en$ (see Subsection~\ref{subsect:induction}), therefore if an element $\beta\en$ was followed by a certain beginning element $\gamma\be=\sigma(\beta\en)$, $\gamma\in\A$, then $\beta\en$ would be the only ending element in this cycle, and after application of this step there would be not a single ending element left, so the scheme $\sigma'=\Pi_{\alpha\beta}^{\mathrm{rb}}\sigma$ would cease to be positive. Since the step $\Pi_{\alpha\beta}^{\mathrm{rb}}$ is allowed by assumption, it is not, in particular, ruining positiveness, therefore the element $\beta\en$ is followed by some ending element $\gamma\en=\sigma(\beta\en)$, $\gamma\in\A$, $\gamma\ne\beta$. For the dual scheme $\sigma^*$ this means that $\sigma^*(\beta\be)=\gamma\en$, therefore one of the cycles in it has the form 
$\left[\begin{array}{cc}
\dots & \beta \\
\dots & \gamma \end{array}\right]$ (in two-row notation), and, accordingly to this, one of the commutation relations (\ref{eq:commutation}) has the form $\check F_\beta\circ\dots=F_\gamma\circ\dots$, thus $\check F_\beta$ stands at the first position in the left-hand side of that relation. (Notice, that the two mentioned cycles $\left[\begin{array}{ccc}
\dots & \dots & \dots \\
\dots & \beta\quad \alpha & \dots                                                                                                                                           \end{array}\right]$, $\left[\begin{array}{cc}
\dots & \beta \\
\dots & \gamma \end{array}\right]$ in $\sigma^*$ can coincide.) Accordingly to Theorem~1 of~\cite{Teplinsky23} (which we conveniently cite below in Section~\ref{sect:symmetry}), $\sigma^*=\Pi_{\beta\alpha}^{\mathrm{rb}}\sigma'^*$, where $\sigma'^*$ denotes a rotational scheme of special type that is dual to $\sigma'=\Pi_{\alpha\beta}^{\mathrm{rb}}\sigma$. The scheme $\sigma'^*$ differs from the scheme $\sigma^*$ just by the position of the element $\alpha\en$ in two highlighted cycles, namely: for $\sigma'^*$ they have the form $\left[\begin{array}{ccc}
\dots & \dots & \dots \\
\dots & \beta & \dots                                                                                                                                           \end{array}\right]$, $\left[\begin{array}{cc}
\dots & \beta \\
\dots & \gamma\quad \alpha \end{array}\right]$. Accordingly, the commutation relations for the scheme $\sigma'^*$ differ from the relations for $\sigma^*$ by the following: $F'_\beta$ now stands at the place of the product $F_\alpha\circ F_\beta$, and the relation $\check F_\beta\circ\dots=F_\gamma\circ\dots$ is replaced by $\check F'_\beta\circ\dots=F'_\alpha\circ F'_\gamma\circ\dots$. It is easy to see that the collection $\Fect'=\Pi_{\alpha\beta}^{\mathrm{rb}}\Fect$: $F'_\beta=F_\alpha\circ F_\beta$, $F'_\xi=F_\xi$ for $\xi\ne\beta$ automatically satisfies the new relations, since in this case we have $\check F'_\beta=F_\alpha\circ\check F_\beta$ independently on whether $\beta\be$ is a stationary element or not. Therefore, $\Fect'=\Pi_{\alpha\beta}^{\mathrm{rb}}\Fect$ is indeed a commuting collection satisfying the commutation relations of the form (\ref{eq:commutation}) for the rotational scheme of special type $\sigma'=\Pi_{\alpha\beta}^{\mathrm{rb}}\sigma$, with the same stationary elements $\omega_i\be$, $1\le i\le N$, and the same sizes of breaks $\nu_i$, $1\le i\le N$.

For other steps the statement is proved similarly, additionally taking in consideration the equalities $\check F_\alpha=F_\alpha$, $\check F'_\alpha=F'_\alpha$ for all three steps $\Pi_{\alpha\beta}^{\mathrm{re}}$, $\Pi_{\alpha\beta}^{\mathrm{lb}}$, $\Pi_{\alpha\beta}^{\mathrm{le}}$, and also the equalities $\check F_\beta=F_\beta$, $\check F'_\beta=F'_\beta$ for the step $\Pi_{\alpha\beta}^{\mathrm{le}}$, which are all implied by the assumption that the steps are not prohibited.

Proposition is proved
\end{proof}

\subsection{On the possibility of non-existence of an allowed step for a given interval exchange transformation}
\label{subsect:counterexample}

We would be happy to claim that for a dynamical system of linear-fractional interval exchange transformation corresponding to a certain commuting collection there always exists an elementa\-ry induction step that is at the same time applicable to it in the sense of Subsection~\ref{subsect:induction} and allowed in the sense of Subsection~\ref{subsect:prohibited}, but this statement is wrong, which is shown by the following couter-example.

\subsubsection*{Example}

For $N=2$, $\A=\{\omega_1,\omega_2,\alpha, \beta\}$, consider the following rotational scheme of special type with stationary elements $\omega_1\be$ and $\omega_2\be$: 
$$
\sigma=\{(\omega_1\be,\alpha\en),(\beta\be,\alpha\be,\omega_2\be, \omega_1\en,\beta\en,\omega_2\en)\}=\left\{
\left[\begin{array}{c}
\omega_1 \\
\alpha \end{array}\right],
\left[\begin{array}{ccc}
\beta & \alpha & \omega_2\\
\omega_2 & \beta & \omega_1 \end{array}\right]
\right\}.
$$
It is easy to check that this is indeed a rotational scheme of special type by its definition in Subsection~\ref{subsect:commuting_collections}. The dual scheme by its definition in Subsection~\ref{subsect:definitions} is
$$
\sigma^*=\{(\alpha\be,\omega_1\be,\beta\en),(\omega_2\be,\beta\be,\omega_2\en,\omega_1\en,\alpha\en)\}=\left\{
\left[\begin{array}{c}
\alpha\quad\omega_1 \\
\beta \end{array}\right],
\left[\begin{array}{c}
\omega_2\quad\beta\\
\alpha\quad\omega_1\quad\omega_2 \end{array}\right]
\right\}.
$$

For this scheme $\sigma$, exactly four elementary induction steps are defined, namely $\Pi_{\beta\omega_2}^{\mathrm{lb}}$, $\Pi_{\beta\omega_2}^{\mathrm{le}}$, $\Pi_{\omega_2\omega_1}^{\mathrm{rb}}$, $\Pi_{\omega_2\omega_1}^{\mathrm{re}}$ (see Subsection~\ref{subsect:induction}), and only two of them, namely $\Pi_{\beta\omega_2}^{\mathrm{lb}}$ and $\Pi_{\omega_2\omega_1}^{\mathrm{rb}}$, are not prohibited (see Subsection~\ref{subsect:prohibited}).

Consider a (linear) interval exchange transformation $(\sigma,\xect)$ (see Subsection~\ref{subsect:definitions}) with the following vector of endpoints $\xect$:
\begin{gather*}
x_{\omega_1\be}=0,\quad x_{\omega_2\be}=0,\quad x_{\alpha\be}=-3,\quad x_{\beta\be}=-4,\\
x_{\omega_1\en}=2,\quad x_{\omega_2\en}=-2,\quad x_{\alpha\en}=3,\quad x_{\beta\en}=-1.
\end{gather*}
It is easy to check that the components of the vector $\xect$ do satisfy the relations (\ref{eq:endpoints_relation}), and also to calculate the corresponding vector of lengths $\vect$:
$$
v_{\omega_1}=3,\quad v_{\omega_2}=2,\quad v_{\alpha}=3,\quad v_{\beta}=1.
$$

Let us look at this interval exchange transformation as at a nonlinear one $(\sigma,\xect,\Fect)$ (see Subsection.~\ref{subsect:commuting_collections}) that corresponds to the commuting collection $\Fect$ (with sizes of breaks $\nu_1=\nu_2=1$) of linear algebraic functions, elements of $\psl$:
$$
F_{\omega_1}=\left(\begin{array}{cc}
1 & -1\\
0 & 1 \end{array}\right),\quad F_{\omega_2}=\left(\begin{array}{cc}
1 & -4\\
0 & 1 \end{array}\right),\quad F_{\alpha}=\left(\begin{array}{cc}
1 & 3\\
0 & 1 \end{array}\right),\quad F_{\beta}=\left(\begin{array}{cc}
1 & 3\\
0 & 1 \end{array}\right).
$$

Since $v_{\beta}<v_{\omega_2}$, while $v_{\omega_2}<v_{\omega_1}$, the both allowed steps $\Pi_{\beta\omega_2}^{\mathrm{lb}}$, $\Pi_{\omega_2\omega_1}^{\mathrm{rb}}$ are not appicable (see Subsection~\ref{subsect:induction}) to the given interval exchange transformation. And therefore, for the given commuting collection $\Fect$, to which the interval exchange transformation $(\sigma,\xect,\Fect)$ corresponds, not a single elementary induction step is at the same time allowed and applicable.

(The example we presented above considers a linear interval exchange transformation, however by a small perturbation in the space of commuting collections it obviously can be transformed into a nonlinear one, and with non-unit sizes of breaks.)

\section{Duality for commuting collections}
\label{sect:duality}

Let a rotational scheme of special type $\sigma$ and sizes of breaks
$\nu_1,\dots,\nu_N$ be given. In this section we present an algorithm which,
for a given commuting collection of algebraic linear-fractional functions
$\Fect=(F_\alpha)_{\alpha\in\A}$, considered as elements of $\psl$ and
satisfying the commutation relations (\ref{eq:commutation}), constructs the
dual commuting collection $\Fect^*=(F^*_\alpha)_{\alpha\in\A}$, and describe
their mutual properties. We also define auxiliary arrays of
\emph{nodal} linear-fractional functions
$\Phict=(\Phi_{p\alpha})_{1\le p\le4,\alpha\in\A}$, which are in one-to-one
correspondence with commuting collections $\Fect=(F_\alpha)_{\alpha\in\A}$,
and their dual arrays.

\subsection{Patchwork surface of the natural extension}
\label{subsect:patchwork_surface}

For the most visual construction, we use a surface glued from rectangles
whose sides correspond to the rearranging intervals of a given positive
natural extension of a linear IRE with scheme $\sigma$. The general
construction of this surface was described by us in Sect.~10 of
\cite{Teplinsky23} (it generalizes Veech's classical `zipped rectangles'
construction~\cite{Veech82}), but for rotational interval exchanges both the
construction algorithm and its result are considerably simpler, so we
describe them here, calling the result a `patchwork surface'.

Take an arbitrary positive natural extension of an IRE with scheme $\sigma$,
that is, a pair of rotational linear interval exchanges
$[(\sigma,\uect),(\sigma^*,\yect)]$ (see Sect.~9 of \cite{Teplinsky23}).
Since we need this surface only to visualize the construction of the dual
commuting collection, the endpoint vectors
$\uect,\yect\in\R^{\bar\A}$ may be chosen arbitrarily among those allowed by
the schemes $\sigma,\sigma^*$, respectively, provided that the corresponding
length vectors $\vect,\wect\in\R^\A$, constructed according to
(\ref{eq:endpoints_relation}), are positive. Such endpoint vectors always
exist, since $\sigma$ and $\sigma^*$ are positive by definition of a
rotational scheme. We imagine the interval exchange $(\sigma,\uect)$ lying
on the abscissa axis (horizontally), and $(\sigma^*,\yect)$ on the ordinate
axis (vertically). For every $\alpha\in\A$, take a (half-open) rectangle
$R_\alpha$ of horizontal size
$v_\alpha=u_{\sigma(\alpha\be)}-u_{\alpha\be}
=u_{\alpha\en}-u_{\sigma(\alpha\en)}$
and vertical size
$w_\alpha=y_{\sigma(\alpha\be)}-y_{\alpha\be}
=y_{\alpha\en}-y_{\sigma(\alpha\en)}$, and identify its lower side with the
beginning interval $[u_{\alpha\be},u_{\sigma(\alpha\be)})$ of
$(\sigma,\uect)$, its upper side with the ending interval
$[u_{\sigma(\alpha\en)},u_{\alpha\en})$, its left side with the beginning
interval $[y_{\alpha\be},y_{\sigma^*(\alpha\be)})$ of
$(\sigma^*,\yect)$, and its right side with the ending interval
$[y_{\sigma^*(\alpha\en)},y_{\alpha\en})$ of $(\sigma^*,\yect)$. (Recall that
$\sigma^*(\alpha\be)=\sigma(\alpha\en)$ and
$\sigma^*(\alpha\en)=\sigma(\alpha\be)$ by definition of the dual scheme.)

As a result of these identifications, the $2N$ rectangles $R_\alpha$,
$\alpha\in\A$, are glued into a geometrically flat toroidal surface without
gaps or overlaps. The vertices of these rectangles form $4N$
\emph{nodal points} (or, briefly, \emph{nodes}), at each of which the
coordinates $u_{\bar\alpha}$ and $y_{\bar\alpha}$ from the interval exchanges
$(\sigma,\uect)$ and $(\sigma^*,\yect)$, $\bar\alpha\in\bar\A$, are glued
together.

The figure below shows an example of a patchwork surface for the rotational
scheme of special type
$\sigma=\left\{
\left[\begin{array}{cc}
\alpha &\omega_1\\
\omega_2&\beta
\end{array}\right],
\left[\begin{array}{cc}
\beta&\omega_2\\
\omega_1&\alpha
\end{array}\right]\right\}$
for $\A=\{\omega_1,\omega_2,\alpha,\beta\}$, $N=2$. Its dual scheme is
$\sigma^*=\left\{
\left[\begin{array}{cc}
\omega_1&\beta\\
\alpha&\omega_2
\end{array}\right],
\left[\begin{array}{cc}
\omega_2&\alpha\\
\beta&\omega_1
\end{array}\right]\right\}$.
All rectangular patches marked identically in the figure are in fact the
same rectangle; there are four rectangles in this surface:
$R_{\omega_1}$, $R_{\omega_2}$, $R_\alpha$, $R_\beta$, corresponding to the
symbols of $\A$.

\vspace{20pt}

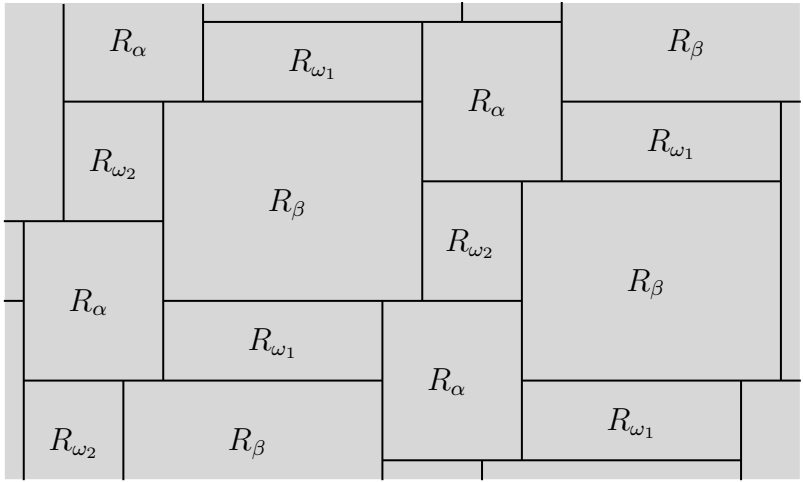
\tikzset{every picture/.style={line width=0.75pt}} 

\begin{tikzpicture}[x=0.75pt,y=0.75pt,yscale=-1,xscale=1]

\draw  [color={rgb, 255:red, 255; green, 255; blue, 255 }  ,draw opacity=1 ][fill={rgb, 255:red, 0; green, 0; blue, 0 }  ,fill opacity=0.16 ][line width=0.75]  (40,50) -- (440,50) -- (440,290) -- (40,290) -- cycle ;
\draw    (70,51) -- (70,160) ;
\draw    (120,100) -- (120,240) ;
\draw    (70,100) -- (250,100) ;
\draw    (140,51) -- (140,100) ;
\draw    (140,60) -- (320,60) ;
\draw    (40,160) -- (120,160) ;
\draw    (40,200) -- (50,200) ;
\draw    (50,160) -- (50,290) ;
\draw    (100,240) -- (100,290) ;
\draw    (50,240) -- (230,240) ;
\draw    (120,200) -- (300,200) ;
\draw    (250,60) -- (250,200) ;
\draw    (320,50) -- (320,140) ;
\draw    (250,140) -- (430,140) ;
\draw    (320,100) -- (440,100) ;
\draw    (270,50) -- (270,60) ;
\draw    (430,100) -- (430,240) ;
\draw    (300,140) -- (300,280) ;
\draw    (230,200) -- (230,290) ;
\draw    (230,280) -- (410,280) ;
\draw    (280,280) -- (280,290) ;
\draw    (410,240) -- (410,290) ;
\draw    (300,240) -- (440,240) ;

\draw (171,142.4) node [anchor=north west][inner sep=0.75pt]    {$R_{\beta }$};
\draw (351,182.4) node [anchor=north west][inner sep=0.75pt]    {$R_{\beta }$};
\draw (71,192.4) node [anchor=north west][inner sep=0.75pt]    {$R_{\alpha }$};
\draw (271,92.4) node [anchor=north west][inner sep=0.75pt]    {$R_{\alpha }$};
\draw (251,232.4) node [anchor=north west][inner sep=0.75pt]    {$R_{\alpha }$};
\draw (151,262.4) node [anchor=north west][inner sep=0.75pt]    {$R_{\beta }$};
\draw (371,62.4) node [anchor=north west][inner sep=0.75pt]    {$R_{\beta }$};
\draw (81,122.4) node [anchor=north west][inner sep=0.75pt]    {$R_{\omega _{2}}$};
\draw (260,162.4) node [anchor=north west][inner sep=0.75pt]    {$R_{\omega _{2}}$};
\draw (61,262.4) node [anchor=north west][inner sep=0.75pt]    {$R_{\omega _{2}}$};
\draw (91,62.4) node [anchor=north west][inner sep=0.75pt]    {$R_{\alpha }$};
\draw (181,72.4) node [anchor=north west][inner sep=0.75pt]    {$R_{\omega _{1}}$};
\draw (161,212.4) node [anchor=north west][inner sep=0.75pt]    {$R_{\omega _{1}}$};
\draw (361,112.4) node [anchor=north west][inner sep=0.75pt]    {$R_{\omega _{1}}$};
\draw (341,252.4) node [anchor=north west][inner sep=0.75pt]    {$R_{\omega _{1}}$};

\end{tikzpicture}

\vspace{20pt}

In Sect.~4.2 of \cite{Teplinsky24} we classified the endpoints of the intervals
for a given scheme $\sigma$ into four types: an endpoint labeled
$\alpha\be$, $\alpha\in\A$, has type L (`left') if
$\sigma^{-1}(\alpha\be)=\beta\en$, and type MB (`middle beginning') if
$\sigma^{-1}(\alpha\be)=\beta\be$, $\beta\in\A$; an endpoint labeled
$\alpha\en$, $\alpha\in\A$, has type R (`right') if
$\sigma^{-1}(\alpha\en)=\beta\be$, and type ME (`middle ending') if
$\sigma^{-1}(\alpha\en)=\beta\en$, $\beta\in\A$.
Clearly, this is a classification of the labels themselves, independent of
the coordinates of the corresponding endpoints, but it is more convenient to
refer to the endpoints: for every interval exchange segment (corresponding to
one cycle of the permutation $\sigma$), its left endpoint has type L, its
right endpoint has type R, all beginning endpoints inside the segment, if any,
have type MB, and all ending endpoints inside it, if any, have type ME. We shall
use the same terminology for the nodes of the patchwork surface. By the
definition of a rotational scheme of special type, among these $4N$
nodes there are exactly $N$ nodes of each of the four types. (Note that, by
the definition of duality, labels of types L, R, MB, ME in the scheme
$\sigma$ are labels of types MB, ME, L, R, respectively, in $\sigma^*$, and
vice versa.)

The $2N$ rectangles have $8N$ right angles, and at each of the $4N$ nodes of
the constructed toroidal surface two particular rectangle corners meet:
at a node $\alpha\be$ of type L, the lower-left corner of $R_\alpha$ and the
upper-left corner of $R_\beta$, where
$\beta\en=\sigma^{-1}(\alpha\be)$, meet, and these two right angles are glued
along the lower side of $R_\alpha$ and the upper side of $R_\beta$; at a
node $\alpha\be$ of type MB, the lower-left corner of $R_\alpha$ and the
lower-right corner of $R_\beta$, where
$\beta\be=\sigma^{-1}(\alpha\be)$, meet, and these are glued along the left side
of $R_\alpha$ and the right side of $R_\beta$; at a node $\alpha\en$ of type
R, the upper-right corner of $R_\alpha$ and the lower-right corner of
$R_\beta$, where $\beta\be=\sigma^{-1}(\alpha\en)$, meet, and these are glued
along the upper side of $R_\alpha$ and the lower side of $R_\beta$; at a
node $\alpha\en$ of type ME, the upper-right corner of $R_\alpha$ and the
upper-left corner of $R_\beta$, where
$\beta\en=\sigma^{-1}(\alpha\en)$, meet, and these are glued along the right side
of $R_\alpha$ and the left side of $R_\beta$. We call the listed pairs of
rectangle corners that meet at the nodes \emph{adjacent}. (Note that every
two adjacent right angles together form a straight angle which, in generic case, is completed to a full angle by another rectangle glued to the
given node along its side.)

\subsection{Array of nodal linear-fractional functions}
\label{subsect:Phi}

For a given commuting collection, define the array
$\Phict=(\Phi_{p\alpha})_{1\le p\le4,\alpha\in\A}$ of $8N$ algebraic
linear-fractional functions
$\Phi_{1\alpha},\Phi_{2\alpha},\Phi_{3\alpha},\Phi_{4\alpha}\in\psl$,
which we place at the lower-left, lower-right, upper-right, and upper-left
corners, respectively, of the rectangles $R_\alpha$, $\alpha\in\A$,
according to the following rules.

Rule 1): for every stationary element $\omega_i\be\in\bar\A$,
$1\le i\le N$, the function $\Phi_{1\omega_i}=\nu_i\Id$ is placed at the
lower-left corner of $R_{\omega_i}$, while the function $\Id$ is placed at
the adjacent corner.

Rule 2): equal functions are placed at two adjacent corners not covered by
Rule~1).

Rule 3): functions placed at the lower corners of $R_\alpha$ are transformed into
the functions placed at its corresponding upper corners under the action of
$F_\alpha$, i.e.
$\Phi_{4\alpha}=F_\alpha\circ\Phi_{1\alpha}$,
$\Phi_{3\alpha}=F_\alpha\circ\Phi_{2\alpha}$, $\alpha\in\A$.

We now show how to fill algorithmically all $8N$ corners of the rectangles
of the patchwork surface with functions.
We go through all cycles of the dual scheme $\sigma^*$, one by one. Each such
cycle contains exactly one stationary element, by the definition of a
rotational scheme of special type, and we start from it, placing the
functions at the two adjacent corners meeting at the corresponding node
according to Rule~1). We then move along the chosen cycle in either
direction and place new functions at the corresponding corners
using Rules~3) and~2) alternately. The commutation relations
(\ref{eq:commutation}) guarantee that, when the cycle closes, the same
function is placed at the corresponding corner whether we traverse the
cycle in the forward or backward direction.

More formally, let $\sigma^*=\{z_1,\dots,z_N\}$, and consider the cycle
$$
z_i=\left[\begin{array}{ccc}
\gamma_1&\dots&\gamma_s\\
\delta_1&\dots&\delta_t
\end{array}\right]
=(\gamma_1\be,\dots,\gamma_s\be,\delta_t\en,\dots,\delta_1\en),
\quad 1\le i\le N
$$
(see Subsection~\ref{subsect:commuting_collections}, where the structure of
rotational schemes of special type is described; since the cycle under
consideration is fixed, we omit the subscript $i$ from $\gamma,\delta,s,t$ for
simplicity). Assume that the stationary element in this cycle is
$\gamma_k\be$, i.e.\ $\omega_i=\gamma_k$, $1\le k\le s$. The cycle $z_i$
corresponds on the surface glued from the rectangles to a vertical segment
consisting, on one side, of the left sides of
$R_{\gamma_1},\dots,R_{\gamma_s}$ in succession from bottom to top, and,
on the other side, of the right sides of
$R_{\delta_1},\dots,R_{\delta_t}$, likewise from bottom to top. The left
corners of the first sequence and the right corners of the second one form a
closed chain of $2s+2t$ corners, and the functions at each pair of neighboring
corners (either adjacent corners or corners of the same rectangle) are
related by one of Rules~1)--3). These $2s+2t$ functions can, for example,
be determined by the following algorithm.

First we place $\Phi_{1\gamma_k}=\nu_i\Id$ at the lower-left corner of
$R_{\gamma_k}$ according to Rule~1). Next we place
$\Phi_{4\gamma_k}=F_{\gamma_k}\circ\Phi_{1\gamma_k}
=F_{\gamma_k}\circ\nu_i\Id$ at its upper-left corner according to Rule~3).
Then we place
$\Phi_{1\gamma_{k+1}}=\Phi_{4\gamma_k}
=F_{\gamma_k}\circ\nu_i\Id$ at the adjacent lower-left corner of
$R_{\gamma_{k+1}}$ according to Rule~2). Continue to move upward like this 
using Rules`2) and~3) alternately and filling the left corners of
$R_{\gamma_k},\dots,R_{\gamma_s}$ successively---until we reach the upper endpoint of the vertical
segment corresponding to $z_i$. The last function written is
$\Phi_{4\gamma_s}
=F_{\gamma_s}\circ\dots\circ F_{\gamma_k}\circ\nu_i\Id$ at the upper-left
corner of $R_{\gamma_s}$. If $k>1$, then we fill the left corners of
$R_{\gamma_1},\dots,R_{\gamma_{k-1}}$ in the similar way, but now moving downward:
$\Phi_{4\gamma_{k-1}}=\Id$ by Rule~1),
$\Phi_{1\gamma_{k-1}}=F^{-1}_{\gamma_{k-1}}$ by Rule~3),
$\Phi_{4\gamma_{k-2}}=F^{-1}_{\gamma_{k-1}}$ by Rule~2), and so on,
alternating Rules~2) and~3), until
$\Phi_{1\gamma_1}=F^{-1}_{\gamma_1}\circ\dots\circ
F^{-1}_{\gamma_{k-1}}$ is placed.

Next we fill the right corners of $R_{\delta_1},\dots,R_{\delta_t}$,
moving upward from the lower endpoint of the vertical segment. First we place
$\Phi_{2\delta_1}=\Phi_{1\gamma_1}
=F^{-1}_{\gamma_1}\circ\dots\circ F^{-1}_{\gamma_{k-1}}$ by Rule~2) if
$k>1$, and $\Phi_{2\delta_1}=\Id$ by Rule~1) if $k=1$. (The first formula
is also valid for $k=1$, since the empty composition is $\Id$.) Then, by
alternating Rules~2) and~3) again, we place the remaining functions at
the right corners up to
$\Phi_{3\delta_t}
=F_{\delta_t}\circ\dots\circ F_{\delta_1}\circ
F^{-1}_{\gamma_1}\circ\dots\circ F^{-1}_{\gamma_{k-1}}$.

Thus all $2s+2t$ linear-fractional functions corresponding to the chosen
cycle $z_i$ in $\sigma^*$ have been determined:

\noindent $\Phi_{1\gamma_{1}}=F^{-1}_{\gamma_{1}}\circ F^{-1}_{\gamma_{2}}\circ \dots\circ F^{-1}_{\gamma_{k-1}}$,  $\Phi_{4\gamma_{1}}=\Phi_{1\gamma_{2}}=F^{-1}_{\gamma_{2}}\circ\dots\circ F^{-1}_{\gamma_{k-1}}$, \dots, $\Phi_{4\gamma_{k-1}}=\Id$;
$\Phi_{1\gamma_{k}}=\nu_i\Id$, $\Phi_{4\gamma_{k}}=\Phi_{1\gamma_{k+1}}=F_{\gamma_{k}}\circ\nu_i\Id$, \dots, $\Phi_{4\gamma_{s}}=F_{\gamma_{s}}\circ\dots\circ F_{\gamma_{k}}\circ\nu_i\Id$; 
$\Phi_{2\delta_1}=F^{-1}_{\gamma_{1}}\circ\dots\circ F^{-1}_{\gamma_{k-1}}$, $\Phi_{3\delta_1}=\Phi_{2\delta_2}=F_{\delta_1}\circ F^{-1}_{\gamma_{1}}\circ\dots\circ F^{-1}_{\gamma_{k-1}}$, \dots, $\Phi_{3\delta_t}=F_{\delta_t}\circ\dots F_{\delta_1}\circ F^{-1}_{\gamma_{1}}\circ\dots\circ F^{-1}_{\gamma_{k-1}}$.

\noindent
(Of course, these formulas could be written down straight away, but we believe that the visual algorithm described above explains the very logic of their construction.)

It is easy to see that the functions
$\Phi_{4\gamma_s}=F_{\gamma_s}\circ\dots\circ F_{\gamma_k}\circ\nu_i\Id$
and
$\Phi_{3\delta_t}=F_{\delta_t}\circ\dots\circ F_{\delta_1}\circ
F^{-1}_{\gamma_1}\circ\dots\circ F^{-1}_{\gamma_{k-1}}$, written at the
adjacent corners near the upper endpoint of the vertical segment corresponding
to $z_i$, are equal by the commutation relations (\ref{eq:commutation});
hence the algorithm is consistent. Applying it successively to every cycle
in $\sigma^*$ produces, in a constructive way, all $8N$ functions
$\Phi_{p\alpha}\in\psl$, $1\le p\le4$, $\alpha\in\A$, satisfying Rules~
1)--3). Since Rule~1) fixes the functions $\Phi_{1\omega_i}$,
$1\le i\le N$, the array $\Phict$ is unique.

The argument above also works in the opposite direction: if Rules~1)--3) hold for an array $\Phict$ and some collection $\Fect$ of algebraic
linear-fractional functions, then the latter satisfies the commutation
relations (\ref{eq:commutation}). Since the functions in $\Phict$
correspond to the nodal points of the patchwork surface (two functions per
node), we call them \emph{nodal functions}. Thus we have constructively
proved the following statement (with all components of $\Fect$ and
$\Phict$ seen as elements of $\psl$).

\begin{proposition}\label{prop:Phi&F}
For a given commuting collection $\Fect$ satisfying (\ref{eq:commutation}),
there exists a unique array of nodal functions $\Phict$ such that the
functions in $\Phict$ and $\Fect$ satisfy Rules~1)--3).
If the functions in a given array $\Phict$ and a collection $\Fect$ satisfy
Rules~1)--3), then $\Fect$ is a commuting collection satisfying
(\ref{eq:commutation}).
\end{proposition}

Effectively, we have shown that commuting collections and arrays of nodal
functions are equivalent mathematical objects, being in one-to-one
correspondence: a given collection $\Fect$ can be algorithmically converted
into an array $\Phict$, and vice versa, for a given rotational scheme of
special type with a given set of stationary elements and given sizes of
breaks.

\subsection{Dual surface, dual array of nodal functions, and dual commuting collection}
\label{subsect:dual_surface}

The natural extension dual to $[(\sigma,\uect),(\sigma^*,\yect)]$ is naturally $[(\sigma^*,\yect),(\sigma,\uect)]$, and
the dual rectangles $R^*_\alpha$, $\alpha\in\A$, are the same rectangles
$R_\alpha$, $\alpha\in\A$, respectively, whose horizontal sides became vertical and vertical sides became horizontal. The dual rectangles are glued along their sides into a toroidal surface in exactly the same way as the original ones. Thus
both the rectangular patches and the surface formed by them for the
dual natural extension $[(\sigma^*,\yect),(\sigma,\uect)]$ are obtained from
the patchwork surface for the original natural extension
$[(\sigma,\uect),(\sigma^*,\yect)]$ by reflection across the
\emph{identity line} $y=u$. (Recall that in this construction the sizes of
the rectangles are irrelevant for us; only their mutual arrangement on the
toroidal surface matters.)

Let a commuting collection $\Fect$ satisfying (\ref{eq:commutation}) be
given. By Proposition~\ref{prop:Phi&F}, there exists an array $\Phict$,
constructed according to the algorithm described in
Subsection~\ref{subsect:Phi}, whose functions satisfy Rules 1)--3).

Define the \emph{dual} to $\Phict=(\Phi_{p\alpha})_{1\le p\le4,\alpha\in\A}$ array of $8N$ functions
$\Phict^*=(\Phi^*_{p\alpha})_{1\le p\le4,\alpha\in\A}$ by the formulas
$\Phi^*_{1\alpha}=\Phi^\intercal_{1\alpha}$, $\Phi^*_{2\alpha}=\Phi^\intercal_{4\alpha}$, $\Phi^*_{3\alpha}=\Phi^\intercal_{3\alpha}$, $\Phi^*_{4\alpha}=\Phi^\intercal_{2\alpha}$, $\alpha\in\A$, where ${}^\intercal$ denotes matrix transposition, which is an involutive operation
on $\psl$. Thus, if we visualize the nodal functions $\Phi_{p\alpha}$ as
written at the corners of the rectangles $R_\alpha$, then the duality transformation
reflects the rectangles across the identity line and transposes the functions
at their corners as elements of $\psl$.

Consider now Rules 1)--3) from Subsection~\ref{subsect:Phi}, but for the
dual scheme $\sigma^*$ and the dual patchwork surface. By construction,
$\Phi^*_{1\omega_i}=\Phi^\intercal_{1\omega_i}
=(\nu_i\Id)^\intercal=\nu_i\Id$, $1\le i\le N$, and since adjacent corners
remain adjacent on the dual surface, Rules~1) and~2) hold for $\Phict^*$.
By Rule~3),
$\Phi_{4\alpha}\circ\Phi^{-1}_{1\alpha}
=\Phi_{3\alpha}\circ\Phi^{-1}_{2\alpha}$, hence $\Phi^*_{4\alpha}\circ(\Phi^*_{1\alpha})^{-1}=\Phi^\intercal_{2\alpha}\circ(\Phi^\intercal_{1\alpha})^{-1}=(\Phi^{-1}_{1\alpha}\circ\Phi_{2\alpha})^\intercal=(\Phi^{-1}_{4\alpha}\circ\Phi_{3\alpha})^\intercal=\Phi^\intercal_{3\alpha}\circ(\Phi^\intercal_{4\alpha})^{-1}=\Phi^*_{3\alpha}\circ(\Phi^*_{2\alpha})^{-1}$, $\alpha\in\A$.
We put
$F^*_\alpha=\Phi^*_{4\alpha}\circ(\Phi^*_{1\alpha})^{-1}
=\Phi^*_{3\alpha}\circ(\Phi^*_{2\alpha})^{-1}$, and now Rule~3), in the form
$\Phi^*_{4\alpha}=F^*_\alpha\circ\Phi^*_{1\alpha}$,
$\Phi^*_{3\alpha}=F^*_\alpha\circ\Phi^*_{2\alpha}$, $\alpha\in\A$, also
holds for the functions in $\Phict^*$ and
$\Fect^*=(F^*_\alpha)_{\alpha\in\A}$.

By Proposition~\ref{prop:Phi&F}, the set $\Fect^*$ defined above is a
commuting collection with the commutation relations corresponding to the
dual scheme $\sigma^*$ and the same sizes of breaks
$\nu_1,\dots,\nu_N$, namely
\begin{equation*}
\tilde F^*_{\alpha_{m_i i}}\circ\dots\circ\tilde F^*_{\alpha_{1i}}
=
F^*_{\beta_{n_i i}}\circ\dots\circ F^*_{\beta_{1i}},
\quad1\le i\le N,
\end{equation*}
where $\tilde F^*_{\omega_i}=F^*_{\omega_i}\circ\nu_i\Id$,
$1\le i\le N$, and $\tilde F^*_\alpha=F^*_\alpha$ for
$\alpha\notin\{\omega_1,\dots,\omega_N\}$, while
$c_i=\left[\begin{array}{ccc}
\alpha_{1i}&\dots&\alpha_{m_i i}\\
\beta_{1i}&\dots&\beta_{n_i i}
\end{array}\right]$
is the cycle in $\sigma$ containing the stationary element $\omega_i\be$,
$1\le i\le N$. We call the commuting collection $\Fect^*$ constructed by
the algorithm described above \emph{dual} to the commuting collection
$\Fect$.

The symmetry of our construction for the nodal functions implies that
$\Phict$ is dual to $\Phict^*$, and hence that $\Fect$ is dual to
$\Fect^*$ (formally, $\Phict^{**}=\Phict$, $\Fect^{**}=\Fect$). Thus the
\emph{duality operation}, again denoted by $\Inv$ (so that
$\Phict^*=\Inv\Phict$, $\Fect^*=\Inv\Fect$), is an involution both on arrays
of nodal functions and on commuting collections:
$\Inv^2=\Id$.

\section{Time-reversing symmetry for induction on commuting collections}
\label{sect:symmetry}

In the work~\cite{Teplinsky23}, a result that is key for the theory of interval rearrangement ensembles was formulated and proved, which establishes the time-reversing symmetry for induction of Rauzy--Veech type on IRE schemes, so let us quote it in exact words (only having removed a single internal reference): 

\begin{theorem_nonumber}[\cite{Teplinsky23}]
Each of the following equalities holds on the set of all those IRE schemes for which its left-hand side is defined:
\begin{gather*}
(\Pi_{\alpha\beta}^{\mathrm{rb}})^{-1}=\Inv\circ\Pi_{\beta\alpha}^{\mathrm{rb}}\circ\Inv,
\\
(\Pi_{\alpha\beta}^{\mathrm{re}})^{-1}=\Inv\circ\Pi_{\alpha\beta}^{\mathrm{lb}}\circ\Inv,
\\
(\Pi_{\alpha\beta}^{\mathrm{lb}})^{-1}=\Inv\circ\Pi_{\alpha\beta}^{\mathrm{re}}\circ\Inv,
\\
(\Pi_{\alpha\beta}^{\mathrm{le}})^{-1}=\Inv\circ\Pi_{\beta\alpha}^{\mathrm{le}}\circ\Inv.
\end{gather*}
\end{theorem_nonumber}

In this section we extend this result to the commuting collections of algebraic linear-fractional functions defined by us in Subsection~\ref{subsect:commuting_collections}, i.e.\ the elements of the space $\psl$ related to each other by commutation relations of the form (\ref{eq:commutation}). Let us remind that such commutation relations determine the set of stationary elements, the sizes of breaks, and also the rotational scheme of special type itself, corresponding to the given commuting collection.

\subsection{Formulation of the result in terms of allowed steps}
\label{subsect:theorem1}

\begin{theorem}\label{theorem:1}
Each of the following equalities holds on the set of all those commuting collections, for which the elementary induction steps in this equality are allowed for the corresponding rotational schemes of special type:
\begin{gather*}
(\Pi_{\alpha\beta}^{\mathrm{rb}})^{-1}=\Inv\Pi_{\beta\alpha}^{\mathrm{rb}}\Inv,
\\
(\Pi_{\alpha\beta}^{\mathrm{re}})^{-1}=\Inv\Pi_{\alpha\beta}^{\mathrm{lb}}\Inv,
\\
(\Pi_{\alpha\beta}^{\mathrm{lb}})^{-1}=\Inv\Pi_{\alpha\beta}^{\mathrm{re}}\Inv,
\\
(\Pi_{\alpha\beta}^{\mathrm{le}})^{-1}=\Inv\Pi_{\beta\alpha}^{\mathrm{le}}\Inv.
\end{gather*}
\end{theorem}

We deliberately formulated this theorem in a form similar to the formulation of the theorem from~\cite{Teplinsky23} given above, because we consider it to be the most visual and practical. (We removed the composition signs from the equalities in order to avoid possible confusion with compositions of linear-fractional functions, as in the commutation relations, for ex.) But in this form our formulation requires some additional explanation in view of strengthening the notion of steps being defined to the notion of them being allowed (according to Subsection~\ref{subsect:prohibited}), which explanation we will give now.

For shortening the exposition, let us rewrite the equalities from the formulations of the theorems in a slightly different form. Let $\Pi$ be one of the elementary induction steps $\Pi_{\alpha\beta}^{\mathrm{rb}}$, $\Pi_{\alpha\beta}^{\mathrm{re}}$, $\Pi_{\alpha\beta}^{\mathrm{lb}}$, and $\Pi_{\alpha\beta}^{\mathrm{le}}$, $\alpha,\beta\in\A$. Let us denote by $\Pi^*$ the elementary induction step $\Pi_{\beta\alpha}^{\mathrm{rb}}$, $\Pi_{\alpha\beta}^{\mathrm{lb}}$, $\Pi_{\alpha\beta}^{\mathrm{re}}$ and $\Pi_{\beta\alpha}^{\mathrm{le}}$ respectively, and call it \emph{dual} to $\Pi$ (it is easy to check that $\Pi^{**}=\Pi$, i.~e. the defined duality on steps is an involution). In the introduced notation, the four equalities from the formulations of both theorems can be replaced by the following single equality
$$
\Pi^{-1}=\Inv\Pi^*\Inv.
$$

In the theorem from~\cite{Teplinsky23} quoted above, this equality concerns the action of the elementary induction steps $\Pi$, $\Pi^*$ and the duality involution $\Inv$ on IRE schemes, and hence the theorem actually states that the equality
$$
\Pi^*\Inv\Pi\sigma=\Inv\sigma
$$
holds for all IRE schemes, on which the given induction steps are defined---that is, the step $\Pi$ is defined on the scheme $\sigma$, while the step $\Pi^*$ is defined on the scheme $\sigma'^*=\Inv\Pi\sigma$ (in fact, these two conditions are equivalent, see the remark at the end of this section). 

And the formulation of the present Theorem~\ref{theorem:1} actually means the following: the equality
$$
\Pi^*\Inv\Pi\Fect=\Inv\Fect
$$
holds for all commuting collections $\Fect$, for which the elementary induction steps $\Pi$ and $\Pi^*$ are allowed for the corresponding rotational schemes of special type---that is, the step $\Pi$ is allowed for the scheme $\sigma$, which corresponds to the collection $\Fect$, while the step $\Pi^*$ is allowed for the scheme $\sigma'^*=\Inv\Pi\sigma$, which corresponds to the collection $\Fect'^*=\Inv\Pi\Fect$.

This result can also be equivalently reformulated in terms of the arrays of nodal functions defined in Subsection~\ref{subsect:Phi}: the equality
$$
\Pi^*\Inv\Pi\Phict=\Inv\Phict
$$
holds for all arrays $\Phict$, for which the steps $\Pi$ and $\Pi^*$ are allowed for the corresponding schemes---and in these terms we will be proving our theorem in Subsection~\ref{subsect:proof}.

Let us note that the elementary induction steps $\Pi$ and $\Pi^*$ are always either both defined on the schemes $\sigma$ and $\sigma'^*$ respectively, or both not defined (this is easy to check according to the statement quoted in Subsection~\ref{subsect:prohibited} and the definition of duality on schemes given in Subsection~\ref{subsect:definitions}). It is also true that $\Pi$ and $\Pi^*$ are always either both prohibited for the given set of stationary elements, or both not prohibited (this follows directly from the list of prohibited steps in Subsection~\ref{subsect:prohibited}). But with being ruining positiveness the situation is different: the step $\Pi^*$ is never ruining positiveness for the scheme $\sigma'^*$ (simply because the scheme $\sigma$ is rotational by assumption, hence the dual scheme $\sigma^*=\Pi^*\sigma'^*$ is always positive), and therefore, in this consideration, only the step $\Pi$ can turn out to be ruining positiveness.

\subsection{Induction on nodal functions}
\label{subsect:induction_Phi}

As was noted in Subsection~\ref{subsect:Phi}, commuting collections are in one-to-one correspondence with arrays of nodal functions, and the definition of the duality operation on commuting collections in Subsection~\ref{subsect:dual_surface}\, was given by us through the definition of duality on arrays of nodal functions. We will prove Theorem~\ref{theorem:1} in terms of arrays $\Phict$ as well, rather than in terms of collections $\Fect$. For this we will need to learn how do allowed (according to Subsection~\ref{subsect:prohibited}) elementary induction steps act on nodal functions.

For proper visualization it is necessary to choose the patchwork surface of the natural extension $[(\sigma,\uect), (\sigma^*,\yect)]$ described in Subsection~\ref{subsect:patchwork_surface} in such a way that its image under the step $\Pi$ is also a patchwork surface---that is, so that for the four vectors of endpoints $\uect, \uect', \yect, \yect'\in\R^{\bar\A}$ in the equality
$$
[(\sigma',\uect'), (\sigma'^*,\yect')]=\Pi[(\sigma,\uect), (\sigma^*,\yect)]=[\Pi(\sigma,\uect), (\Pi^*)^{-1}(\sigma^*,\yect)],
$$
all four corresponding vectors of lengths $\vect, \vect', \wect, \wect'\in\R^\A$ need to be positive. And this can always be done, because the step $\Pi$ is allowed by assumption and hence it is not ruining positiveness. It is sufficient to choose for the positive scheme $\sigma'$ any vector of endpoints $\uect'$ allowed for it, with a positive vector of lengths $\vect'$, and for the positive scheme $\sigma^*$ any vector of endpoints $\yect$ allowed for it, with a positive vector of lengths $\wect$. Then the vectors $\vect$ and $\wect'$ will also automatically be positive, since $(\sigma,\uect)=\Pi^{-1}(\sigma',\uect')$, $(\sigma'^*,\yect')=(\Pi^*)^{-1}(\sigma^*,\yect)$, and the reverse induction steps only add lengths, always transforming positive vectors of lengths into positive ones. It is under these assumptions that we will visualize the arrays of nodal functions on patchwork surfaces in the four cases that we now consider.

\subsubsection*{Case $\Pi=\Pi_{\alpha\beta}^{\mathrm{rb}}$, $\alpha,\beta\in\A$.}

Let us describe this case in detail; the rest are considered similarly. We begin with the fact that the vectors of lengths $\vect,\wect$ under the action of the step $\Pi_{\alpha\beta}^{\mathrm{rb}}$ transform as follows: $v'_\alpha=v_\alpha-v_\beta$, $w'_\beta=w_\beta+w_\alpha$, while all the other lengths remain unchanged (the vectors of lengths $\vect, \vect', \wect, \wect'$ are positive according to our choice of the natural extension described above, by which the patchwork surface is constructed). The rectangular patches $R_\alpha$ and $R_\beta$ have dimensions $v_\alpha\times w_\alpha$ and $v_\beta\times w_\beta$, respectively. These rectangles are mutually positioned on the surface so that the lower right corner of $R_\alpha$ and the upper right corner of $R_\beta$ are adjacent. The rectangles $R'_\alpha$ and $R'_\beta$ have dimensions $v'_\alpha\times w'_\alpha$ and $v'_\beta\times w'_\beta$, that is, $(v_\alpha-v_\beta)\times w_\alpha$ and $v_\beta\times(w_\beta+w_\alpha)$, respectively, and they are joined so that the upper right corner of $R'_\alpha$ and the upper left corner of $R'_\beta$ are adjacent, meeting at one node. It is easy to see that $R_\alpha\cup R_\beta=R'_\alpha\cup R'_\beta$, while all the other rectangles in the patchwork surface remain unchanged under the action of the step $\Pi_{\alpha\beta}^{\mathrm{rb}}$. In the figure below we show only that part of our surface, which is the union of these two patches. It is also easy to see that under the action of the step $\Pi_{\alpha\beta}^{\mathrm{rb}}$ exactly one node on the patchwork surface disappears, and exactly one new node appears (in the figure these two nodes are marked with small circles around them), while all the other nodes stay in their places.

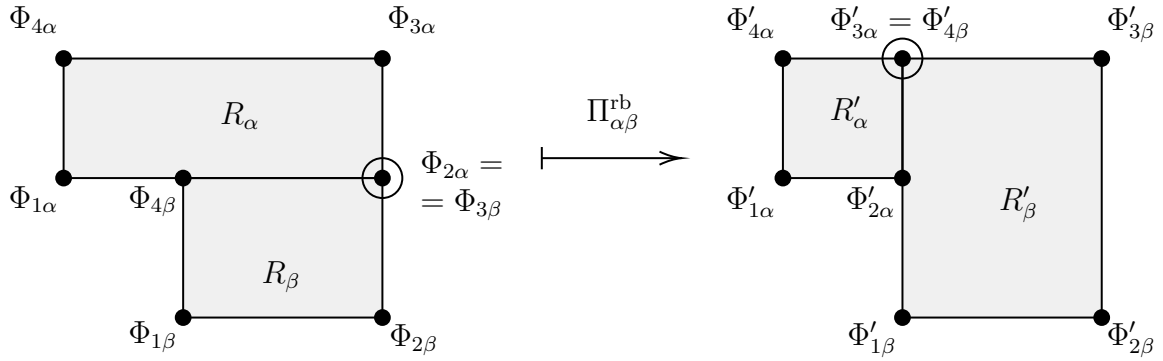
\vspace{20pt}

\begin{tikzpicture}[x=0.75pt,y=0.75pt,yscale=-1,xscale=1]

\draw  [fill={rgb, 255:red, 0; green, 0; blue, 0 }  ,fill opacity=0.06 ] (80,110) -- (240,110) -- (240,170) -- (80,170) -- cycle ;
\draw  [fill={rgb, 255:red, 0; green, 0; blue, 0 }  ,fill opacity=0.06 ] (140,170) -- (240,170) -- (240,240) -- (140,240) -- cycle ;
\draw  [fill={rgb, 255:red, 0; green, 0; blue, 0 }  ,fill opacity=0.06 ] (441,110) -- (501,110) -- (501,170) -- (441,170) -- cycle ;
\draw  [fill={rgb, 255:red, 0; green, 0; blue, 0 }  ,fill opacity=0.06 ] (501,110) -- (601,110) -- (601,240) -- (501,240) -- cycle ;
\draw   (230,170) .. controls (230,164.48) and (234.48,160) .. (240,160) .. controls (245.52,160) and (250,164.48) .. (250,170) .. controls (250,175.52) and (245.52,180) .. (240,180) .. controls (234.48,180) and (230,175.52) .. (230,170) -- cycle ;
\draw   (491,110) .. controls (491,104.48) and (495.48,100) .. (501,100) .. controls (506.52,100) and (511,104.48) .. (511,110) .. controls (511,115.52) and (506.52,120) .. (501,120) .. controls (495.48,120) and (491,115.52) .. (491,110) -- cycle ;
\draw  [fill={rgb, 255:red, 0; green, 0; blue, 0 }  ,fill opacity=1 ] (76.3,110) .. controls (76.3,107.96) and (77.96,106.3) .. (80,106.3) .. controls (82.04,106.3) and (83.7,107.96) .. (83.7,110) .. controls (83.7,112.04) and (82.04,113.7) .. (80,113.7) .. controls (77.96,113.7) and (76.3,112.04) .. (76.3,110) -- cycle ;
\draw  [fill={rgb, 255:red, 0; green, 0; blue, 0 }  ,fill opacity=1 ] (236.3,110) .. controls (236.3,107.96) and (237.96,106.3) .. (240,106.3) .. controls (242.04,106.3) and (243.7,107.96) .. (243.7,110) .. controls (243.7,112.04) and (242.04,113.7) .. (240,113.7) .. controls (237.96,113.7) and (236.3,112.04) .. (236.3,110) -- cycle ;
\draw  [fill={rgb, 255:red, 0; green, 0; blue, 0 }  ,fill opacity=1 ] (76.3,170) .. controls (76.3,167.96) and (77.96,166.3) .. (80,166.3) .. controls (82.04,166.3) and (83.7,167.96) .. (83.7,170) .. controls (83.7,172.04) and (82.04,173.7) .. (80,173.7) .. controls (77.96,173.7) and (76.3,172.04) .. (76.3,170) -- cycle ;
\draw  [fill={rgb, 255:red, 0; green, 0; blue, 0 }  ,fill opacity=1 ] (236.3,170) .. controls (236.3,167.96) and (237.96,166.3) .. (240,166.3) .. controls (242.04,166.3) and (243.7,167.96) .. (243.7,170) .. controls (243.7,172.04) and (242.04,173.7) .. (240,173.7) .. controls (237.96,173.7) and (236.3,172.04) .. (236.3,170) -- cycle ;
\draw  [fill={rgb, 255:red, 0; green, 0; blue, 0 }  ,fill opacity=1 ] (136.3,240) .. controls (136.3,237.96) and (137.96,236.3) .. (140,236.3) .. controls (142.04,236.3) and (143.7,237.96) .. (143.7,240) .. controls (143.7,242.04) and (142.04,243.7) .. (140,243.7) .. controls (137.96,243.7) and (136.3,242.04) .. (136.3,240) -- cycle ;
\draw  [fill={rgb, 255:red, 0; green, 0; blue, 0 }  ,fill opacity=1 ] (236.3,240) .. controls (236.3,237.96) and (237.96,236.3) .. (240,236.3) .. controls (242.04,236.3) and (243.7,237.96) .. (243.7,240) .. controls (243.7,242.04) and (242.04,243.7) .. (240,243.7) .. controls (237.96,243.7) and (236.3,242.04) .. (236.3,240) -- cycle ;
\draw  [fill={rgb, 255:red, 0; green, 0; blue, 0 }  ,fill opacity=1 ] (437.3,110) .. controls (437.3,107.96) and (438.96,106.3) .. (441,106.3) .. controls (443.04,106.3) and (444.7,107.96) .. (444.7,110) .. controls (444.7,112.04) and (443.04,113.7) .. (441,113.7) .. controls (438.96,113.7) and (437.3,112.04) .. (437.3,110) -- cycle ;
\draw  [fill={rgb, 255:red, 0; green, 0; blue, 0 }  ,fill opacity=1 ] (497.3,110) .. controls (497.3,107.96) and (498.96,106.3) .. (501,106.3) .. controls (503.04,106.3) and (504.7,107.96) .. (504.7,110) .. controls (504.7,112.04) and (503.04,113.7) .. (501,113.7) .. controls (498.96,113.7) and (497.3,112.04) .. (497.3,110) -- cycle ;
\draw  [fill={rgb, 255:red, 0; green, 0; blue, 0 }  ,fill opacity=1 ] (597.3,110) .. controls (597.3,107.96) and (598.96,106.3) .. (601,106.3) .. controls (603.04,106.3) and (604.7,107.96) .. (604.7,110) .. controls (604.7,112.04) and (603.04,113.7) .. (601,113.7) .. controls (598.96,113.7) and (597.3,112.04) .. (597.3,110) -- cycle ;
\draw  [fill={rgb, 255:red, 0; green, 0; blue, 0 }  ,fill opacity=1 ] (437.3,170) .. controls (437.3,167.96) and (438.96,166.3) .. (441,166.3) .. controls (443.04,166.3) and (444.7,167.96) .. (444.7,170) .. controls (444.7,172.04) and (443.04,173.7) .. (441,173.7) .. controls (438.96,173.7) and (437.3,172.04) .. (437.3,170) -- cycle ;
\draw  [fill={rgb, 255:red, 0; green, 0; blue, 0 }  ,fill opacity=1 ] (497.3,170) .. controls (497.3,167.96) and (498.96,166.3) .. (501,166.3) .. controls (503.04,166.3) and (504.7,167.96) .. (504.7,170) .. controls (504.7,172.04) and (503.04,173.7) .. (501,173.7) .. controls (498.96,173.7) and (497.3,172.04) .. (497.3,170) -- cycle ;
\draw  [fill={rgb, 255:red, 0; green, 0; blue, 0 }  ,fill opacity=1 ] (497.3,240) .. controls (497.3,237.96) and (498.96,236.3) .. (501,236.3) .. controls (503.04,236.3) and (504.7,237.96) .. (504.7,240) .. controls (504.7,242.04) and (503.04,243.7) .. (501,243.7) .. controls (498.96,243.7) and (497.3,242.04) .. (497.3,240) -- cycle ;
\draw  [fill={rgb, 255:red, 0; green, 0; blue, 0 }  ,fill opacity=1 ] (597.3,240) .. controls (597.3,237.96) and (598.96,236.3) .. (601,236.3) .. controls (603.04,236.3) and (604.7,237.96) .. (604.7,240) .. controls (604.7,242.04) and (603.04,243.7) .. (601,243.7) .. controls (598.96,243.7) and (597.3,242.04) .. (597.3,240) -- cycle ;
\draw  [fill={rgb, 255:red, 0; green, 0; blue, 0 }  ,fill opacity=1 ] (136.3,170) .. controls (136.3,167.96) and (137.96,166.3) .. (140,166.3) .. controls (142.04,166.3) and (143.7,167.96) .. (143.7,170) .. controls (143.7,172.04) and (142.04,173.7) .. (140,173.7) .. controls (137.96,173.7) and (136.3,172.04) .. (136.3,170) -- cycle ;
\draw    (320,160) -- (388,160) ;
\draw [shift={(390,160)}, rotate = 180] [color={rgb, 255:red, 0; green, 0; blue, 0 }  ][line width=0.75]    (10.93,-3.29) .. controls (6.95,-1.4) and (3.31,-0.3) .. (0,0) .. controls (3.31,0.3) and (6.95,1.4) .. (10.93,3.29)   ;
\draw [shift={(320,160)}, rotate = 180] [color={rgb, 255:red, 0; green, 0; blue, 0 }  ][line width=0.75]    (0,5.59) -- (0,-5.59)   ;

\draw (157,130.4) node [anchor=north west][inner sep=0.75pt]    {$R_{\alpha }$};
\draw (178,210.4) node [anchor=north west][inner sep=0.75pt]    {$R_{\beta }$};
\draw (462.99,127.57) node [anchor=north west][inner sep=0.75pt]    {$R'_{\alpha }$};
\draw (547.99,172.79) node [anchor=north west][inner sep=0.75pt]    {$R'_{\beta }$};
\draw (341,126.4) node [anchor=north west][inner sep=0.75pt]    {$\Pi_{\alpha\beta}^{\mathrm{rb}}$};
\draw (51,82.4) node [anchor=north west][inner sep=0.75pt]    {$\Phi _{4\alpha }$};
\draw (241,82.4) node [anchor=north west][inner sep=0.75pt]    {$\Phi _{3\alpha }$};
\draw (51,172.4) node [anchor=north west][inner sep=0.75pt]    {$\Phi _{1\alpha }$};
\draw (111,240.4) node [anchor=north west][inner sep=0.75pt]    {$\Phi _{1\beta }$};
\draw (111,172.4) node [anchor=north west][inner sep=0.75pt]    {$\Phi _{4\beta }$};
\draw (242,241.9) node [anchor=north west][inner sep=0.75pt]    {$\Phi _{2\beta }$};
\draw (251,152.4) node [anchor=north west][inner sep=0.75pt]    {$ \begin{array}{l}
\Phi _{2\alpha } =\\
=\Phi _{3\beta }
\end{array}$};
\draw (411,82.4) node [anchor=north west][inner sep=0.75pt]    {$\Phi '_{4\alpha }$};
\draw (411,172.4) node [anchor=north west][inner sep=0.75pt]    {$\Phi '_{1\alpha }$};
\draw (463,82.4) node [anchor=north west][inner sep=0.75pt]    {$\Phi '_{3\alpha } =\Phi '_{4\beta }$};
\draw (471,172.4) node [anchor=north west][inner sep=0.75pt]    {$\Phi '_{2\alpha }$};
\draw (602,82.4) node [anchor=north west][inner sep=0.75pt]    {$\Phi '_{3\beta }$};
\draw (602,240.4) node [anchor=north west][inner sep=0.75pt]    {$\Phi '_{2\beta }$};
\draw (472,240.4) node [anchor=north west][inner sep=0.75pt]    {$\Phi '_{1\beta }$};

\end{tikzpicture}

\vspace{20pt}

Now let us derive the formulas for the nodal functions of the array $\Phict'=\Pi_{\alpha\beta}^{\mathrm{rb}}\Phict$. Recall that the action of the step $\Pi_{\alpha\beta}^{\mathrm{rb}}$ on a commuting collection $\Fect$ changes only one of its elements according to the formula $F'_\beta=F_\alpha\circ F_\beta$, while the rest remain unchanged (see Subsection~\ref{subsect:commuting_collections}). Let us also recall (see the proof of Proposition~\ref{prop:induction_cc}) that the scheme $\sigma^*$ contained the cycles $\left[\begin{array}{ccc}
\dots & \dots & \dots \\
\dots & \beta\quad \alpha & \dots                                                                                                                                           \end{array}\right]$, $\left[\begin{array}{cc}
\dots & \beta \\
\dots & \gamma \end{array}\right]$, which under the action of the step $\Pi_{\alpha\beta}^{\mathrm{rb}}$ on the scheme $\sigma$ were transformed into the cycles $\left[\begin{array}{ccc}
\dots & \dots & \dots \\
\dots & \beta & \dots                                                                                                                                           \end{array}\right]$, $\left[\begin{array}{cc}
\dots & \beta \\
\dots & \gamma\quad \alpha \end{array}\right]$ in the scheme $\sigma'^*$ (dual to $\sigma'$), where $\gamma\ne\beta$ is the label of the rectangle $R_\gamma$, which adjoins its upper right corner to the upper left corner of $R_\beta$ (i.e., it is that rectangular patch, which complements the angle of size $3\pi/2$ in the part of the patchwork surface shown in the figure above to the full angle; the case $\gamma=\alpha$ is possible, which does not affect our reasoning). Somewhere in each of the cycles in $\sigma'^*$ there is a stationary element, and moving forward and backward from it along the corresponding cycle according to the algorithm for defining nodal functions described in Subsection~\ref{subsect:Phi}, we obtain the equality $\Phi'_{3\beta}=\Phi_{3\alpha}$, as well as the equalities $\Phi'_{i\xi}=\Phi_{i\xi}$ for $i\xi\not\in\{4\beta,3\alpha,2\alpha\}$. For the remaining three functions we have: $\Phi'_{4\beta}=F'_\beta\circ\Phi'_{1\beta}=F_\alpha\circ (F_\beta\circ\Phi_{1\beta})=F_\alpha\circ\Phi_{4\beta}$ according to rule~3) (see Subsection~\ref{subsect:Phi}); $\Phi'_{3\alpha}=\Phi'_{4\beta}=F_\alpha\circ\Phi_{4\beta}$ according to rule~2); $\Phi'_{2\alpha}=(F'_\alpha)^{-1}\circ\Phi'_{3\alpha}=(F_\alpha)^{-1}\circ (F_\alpha\circ\Phi_{4\beta})=\Phi_{4\beta}$ according to rule~3). Thus, we have obtained formulas for all elements of the array $\Phict'$:
$$
\Phi'_{4\beta}=\Phi'_{3\alpha}=F_\alpha\circ\,\Phi_{4\beta},\, \Phi'_{3\beta}=\Phi_{3\alpha},\, \Phi'_{2\alpha}=\Phi_{4\beta},\quad\text{while}\quad\Phi'_{i\xi}=\Phi_{i\xi},\,i\xi\not\in\{4\beta,3\alpha,3\beta,2\alpha\}
$$
(notice that, if necessary, these formulas can easily be rewritten only in terms of nodal functions by substituting the expression $\Phi_{4\alpha}\circ\Phi_{1\alpha}^{-1}$ instead of $F_\alpha$ into the first of them, according to rule~3); the same applies to the remaining three cases considered below).

In three other cases of elementary induction steps, similar considerations  lead to the trans\-formations of the part of the patchwork surface transformed by a given step $\Pi$ we draw in figures below, and to the formulas for the nodal functions of the array $\Phict'=\Pi\Phict$ we list under the figures. (One must take into account the assumption that the step is allowed, as its non-prohibitedness guarantees that the nodal functions at adjacent corners of rectangles meeting at the nodes marked with small circles are equal.)

\subsubsection*{Case $\Pi=\Pi_{\alpha\beta}^{\mathrm{re}}$, $\alpha,\beta\in\A$:}
\nopagebreak

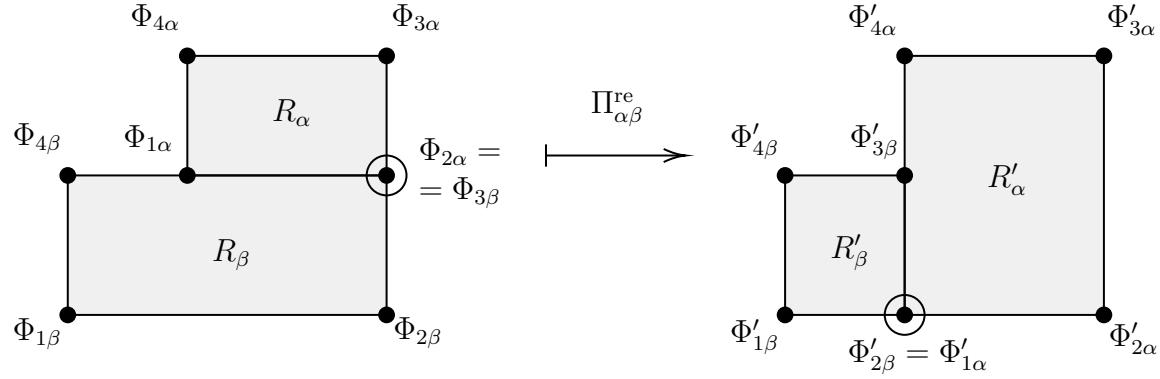
\begin{tikzpicture}[x=0.75pt,y=0.75pt,yscale=-1,xscale=1]

\draw  [fill={rgb, 255:red, 0; green, 0; blue, 0 }  ,fill opacity=0.06 ] (140,110) -- (240,110) -- (240,170) -- (140,170) -- cycle ;
\draw  [fill={rgb, 255:red, 0; green, 0; blue, 0 }  ,fill opacity=0.06 ] (80,170) -- (240,170) -- (240,240) -- (80,240) -- cycle ;
\draw  [fill={rgb, 255:red, 0; green, 0; blue, 0 }  ,fill opacity=0.06 ] (500,110) -- (600,110) -- (600,240) -- (500,240) -- cycle ;
\draw  [fill={rgb, 255:red, 0; green, 0; blue, 0 }  ,fill opacity=0.06 ] (440,170) -- (500,170) -- (500,240) -- (440,240) -- cycle ;
\draw   (230,170) .. controls (230,164.48) and (234.48,160) .. (240,160) .. controls (245.52,160) and (250,164.48) .. (250,170) .. controls (250,175.52) and (245.52,180) .. (240,180) .. controls (234.48,180) and (230,175.52) .. (230,170) -- cycle ;
\draw   (490,240) .. controls (490,234.48) and (494.48,230) .. (500,230) .. controls (505.52,230) and (510,234.48) .. (510,240) .. controls (510,245.52) and (505.52,250) .. (500,250) .. controls (494.48,250) and (490,245.52) .. (490,240) -- cycle ;
\draw  [fill={rgb, 255:red, 0; green, 0; blue, 0 }  ,fill opacity=1 ] (136.3,110) .. controls (136.3,107.96) and (137.96,106.3) .. (140,106.3) .. controls (142.04,106.3) and (143.7,107.96) .. (143.7,110) .. controls (143.7,112.04) and (142.04,113.7) .. (140,113.7) .. controls (137.96,113.7) and (136.3,112.04) .. (136.3,110) -- cycle ;
\draw  [fill={rgb, 255:red, 0; green, 0; blue, 0 }  ,fill opacity=1 ] (236.3,110) .. controls (236.3,107.96) and (237.96,106.3) .. (240,106.3) .. controls (242.04,106.3) and (243.7,107.96) .. (243.7,110) .. controls (243.7,112.04) and (242.04,113.7) .. (240,113.7) .. controls (237.96,113.7) and (236.3,112.04) .. (236.3,110) -- cycle ;
\draw  [fill={rgb, 255:red, 0; green, 0; blue, 0 }  ,fill opacity=1 ] (76.3,240) .. controls (76.3,237.96) and (77.96,236.3) .. (80,236.3) .. controls (82.04,236.3) and (83.7,237.96) .. (83.7,240) .. controls (83.7,242.04) and (82.04,243.7) .. (80,243.7) .. controls (77.96,243.7) and (76.3,242.04) .. (76.3,240) -- cycle ;
\draw  [fill={rgb, 255:red, 0; green, 0; blue, 0 }  ,fill opacity=1 ] (236.3,170) .. controls (236.3,167.96) and (237.96,166.3) .. (240,166.3) .. controls (242.04,166.3) and (243.7,167.96) .. (243.7,170) .. controls (243.7,172.04) and (242.04,173.7) .. (240,173.7) .. controls (237.96,173.7) and (236.3,172.04) .. (236.3,170) -- cycle ;
\draw  [fill={rgb, 255:red, 0; green, 0; blue, 0 }  ,fill opacity=1 ] (136.3,170) .. controls (136.3,167.96) and (137.96,166.3) .. (140,166.3) .. controls (142.04,166.3) and (143.7,167.96) .. (143.7,170) .. controls (143.7,172.04) and (142.04,173.7) .. (140,173.7) .. controls (137.96,173.7) and (136.3,172.04) .. (136.3,170) -- cycle ;
\draw  [fill={rgb, 255:red, 0; green, 0; blue, 0 }  ,fill opacity=1 ] (236.3,240) .. controls (236.3,237.96) and (237.96,236.3) .. (240,236.3) .. controls (242.04,236.3) and (243.7,237.96) .. (243.7,240) .. controls (243.7,242.04) and (242.04,243.7) .. (240,243.7) .. controls (237.96,243.7) and (236.3,242.04) .. (236.3,240) -- cycle ;
\draw  [fill={rgb, 255:red, 0; green, 0; blue, 0 }  ,fill opacity=1 ] (436.3,170) .. controls (436.3,167.96) and (437.96,166.3) .. (440,166.3) .. controls (442.04,166.3) and (443.7,167.96) .. (443.7,170) .. controls (443.7,172.04) and (442.04,173.7) .. (440,173.7) .. controls (437.96,173.7) and (436.3,172.04) .. (436.3,170) -- cycle ;
\draw  [fill={rgb, 255:red, 0; green, 0; blue, 0 }  ,fill opacity=1 ] (496.3,240) .. controls (496.3,237.96) and (497.96,236.3) .. (500,236.3) .. controls (502.04,236.3) and (503.7,237.96) .. (503.7,240) .. controls (503.7,242.04) and (502.04,243.7) .. (500,243.7) .. controls (497.96,243.7) and (496.3,242.04) .. (496.3,240) -- cycle ;
\draw  [fill={rgb, 255:red, 0; green, 0; blue, 0 }  ,fill opacity=1 ] (596.3,110) .. controls (596.3,107.96) and (597.96,106.3) .. (600,106.3) .. controls (602.04,106.3) and (603.7,107.96) .. (603.7,110) .. controls (603.7,112.04) and (602.04,113.7) .. (600,113.7) .. controls (597.96,113.7) and (596.3,112.04) .. (596.3,110) -- cycle ;
\draw  [fill={rgb, 255:red, 0; green, 0; blue, 0 }  ,fill opacity=1 ] (496.3,170) .. controls (496.3,167.96) and (497.96,166.3) .. (500,166.3) .. controls (502.04,166.3) and (503.7,167.96) .. (503.7,170) .. controls (503.7,172.04) and (502.04,173.7) .. (500,173.7) .. controls (497.96,173.7) and (496.3,172.04) .. (496.3,170) -- cycle ;
\draw  [fill={rgb, 255:red, 0; green, 0; blue, 0 }  ,fill opacity=1 ] (496.3,110) .. controls (496.3,107.96) and (497.96,106.3) .. (500,106.3) .. controls (502.04,106.3) and (503.7,107.96) .. (503.7,110) .. controls (503.7,112.04) and (502.04,113.7) .. (500,113.7) .. controls (497.96,113.7) and (496.3,112.04) .. (496.3,110) -- cycle ;
\draw  [fill={rgb, 255:red, 0; green, 0; blue, 0 }  ,fill opacity=1 ] (436.3,240) .. controls (436.3,237.96) and (437.96,236.3) .. (440,236.3) .. controls (442.04,236.3) and (443.7,237.96) .. (443.7,240) .. controls (443.7,242.04) and (442.04,243.7) .. (440,243.7) .. controls (437.96,243.7) and (436.3,242.04) .. (436.3,240) -- cycle ;
\draw  [fill={rgb, 255:red, 0; green, 0; blue, 0 }  ,fill opacity=1 ] (596.3,240) .. controls (596.3,237.96) and (597.96,236.3) .. (600,236.3) .. controls (602.04,236.3) and (603.7,237.96) .. (603.7,240) .. controls (603.7,242.04) and (602.04,243.7) .. (600,243.7) .. controls (597.96,243.7) and (596.3,242.04) .. (596.3,240) -- cycle ;
\draw  [fill={rgb, 255:red, 0; green, 0; blue, 0 }  ,fill opacity=1 ] (76.3,170) .. controls (76.3,167.96) and (77.96,166.3) .. (80,166.3) .. controls (82.04,166.3) and (83.7,167.96) .. (83.7,170) .. controls (83.7,172.04) and (82.04,173.7) .. (80,173.7) .. controls (77.96,173.7) and (76.3,172.04) .. (76.3,170) -- cycle ;
\draw    (320,160) -- (388,160) ;
\draw [shift={(390,160)}, rotate = 180] [color={rgb, 255:red, 0; green, 0; blue, 0 }  ][line width=0.75]    (10.93,-3.29) .. controls (6.95,-1.4) and (3.31,-0.3) .. (0,0) .. controls (3.31,0.3) and (6.95,1.4) .. (10.93,3.29)   ;
\draw [shift={(320,160)}, rotate = 180] [color={rgb, 255:red, 0; green, 0; blue, 0 }  ][line width=0.75]    (0,5.59) -- (0,-5.59)   ;

\draw (181,130.4) node [anchor=north west][inner sep=0.75pt]    {$R_{\alpha }$};
\draw (151,200.4) node [anchor=north west][inner sep=0.75pt]    {$R_{\beta }$};
\draw (540,162.4) node [anchor=north west][inner sep=0.75pt]    {$R'_{\alpha }$};
\draw (461.32,198.12) node [anchor=north west][inner sep=0.75pt]    {$R'_{\beta }$};
\draw (341,126.4) node [anchor=north west][inner sep=0.75pt]    {$\Pi_{\alpha\beta}^{\mathrm{re}}$};
\draw (110,82.4) node [anchor=north west][inner sep=0.75pt]    {$\Phi _{4\alpha }$};
\draw (241,82.4) node [anchor=north west][inner sep=0.75pt]    {$\Phi _{3\alpha }$};
\draw (106.5,143.07) node [anchor=north west][inner sep=0.75pt]    {$\Phi _{1\alpha }$};
\draw (51,240.4) node [anchor=north west][inner sep=0.75pt]    {$\Phi _{1\beta }$};
\draw (51,142.4) node [anchor=north west][inner sep=0.75pt]    {$\Phi _{4\beta }$};
\draw (242,239.7) node [anchor=north west][inner sep=0.75pt]    {$\Phi _{2\beta }$};
\draw (248,147.4) node [anchor=north west][inner sep=0.75pt]    {$ \begin{array}{l}
\Phi _{2\alpha } =\\
=\Phi _{3\beta }
\end{array}$};
\draw (411,142.4) node [anchor=north west][inner sep=0.75pt]    {$\Phi '_{4\beta }$};
\draw (600,82.4) node [anchor=north west][inner sep=0.75pt]    {$\Phi '_{3\alpha }$};
\draw (470,250.4) node [anchor=north west][inner sep=0.75pt]    {$\Phi '_{2\beta } =\Phi '_{1\alpha }$};
\draw (601,242.4) node [anchor=north west][inner sep=0.75pt]    {$\Phi '_{2\alpha }$};
\draw (471,142.4) node [anchor=north west][inner sep=0.75pt]    {$\Phi '_{3\beta }$};
\draw (470,82.4) node [anchor=north west][inner sep=0.75pt]    {$\Phi '_{4\alpha }$};
\draw (411,240.4) node [anchor=north west][inner sep=0.75pt]    {$\Phi '_{1\beta }$};

\end{tikzpicture}
\nopagebreak

$$
\Phi'_{2\beta}=\Phi'_{1\alpha}=F_\beta^{-1}\circ\Phi_{1\alpha},\,\Phi'_{2\alpha}=\Phi_{2\beta},\,\Phi'_{3\beta}=\Phi_{1\alpha},\quad\text{while}\quad\Phi'_{i\xi}=\Phi_{i\xi},\,i\xi\not\in\{2\beta,1\alpha,2\alpha,3\beta\}.
$$

\subsubsection*{Case $\Pi=\Pi_{\alpha\beta}^{\mathrm{lb}}$, $\alpha,\beta\in\A$:}
\nopagebreak

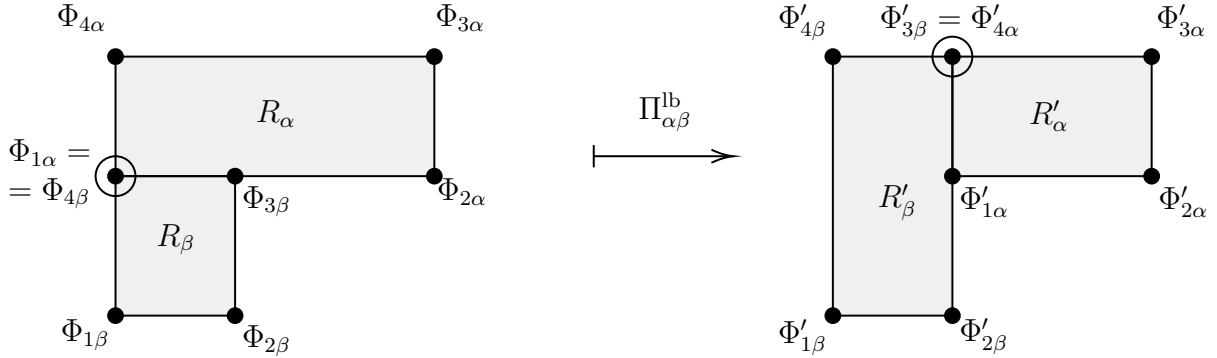
\begin{tikzpicture}[x=0.75pt,y=0.75pt,yscale=-1,xscale=1]

\draw  [fill={rgb, 255:red, 0; green, 0; blue, 0 }  ,fill opacity=0.06 ] (80,110) -- (240,110) -- (240,170) -- (80,170) -- cycle ;
\draw  [fill={rgb, 255:red, 0; green, 0; blue, 0 }  ,fill opacity=0.06 ] (80,170) -- (140,170) -- (140,240) -- (80,240) -- cycle ;
\draw  [fill={rgb, 255:red, 0; green, 0; blue, 0 }  ,fill opacity=0.06 ] (440,110) -- (500,110) -- (500,240) -- (440,240) -- cycle ;
\draw  [fill={rgb, 255:red, 0; green, 0; blue, 0 }  ,fill opacity=0.06 ] (500,110) -- (600,110) -- (600,170) -- (500,170) -- cycle ;
\draw   (70,170) .. controls (70,164.48) and (74.48,160) .. (80,160) .. controls (85.52,160) and (90,164.48) .. (90,170) .. controls (90,175.52) and (85.52,180) .. (80,180) .. controls (74.48,180) and (70,175.52) .. (70,170) -- cycle ;
\draw   (490,110) .. controls (490,104.48) and (494.48,100) .. (500,100) .. controls (505.52,100) and (510,104.48) .. (510,110) .. controls (510,115.52) and (505.52,120) .. (500,120) .. controls (494.48,120) and (490,115.52) .. (490,110) -- cycle ;
\draw  [fill={rgb, 255:red, 0; green, 0; blue, 0 }  ,fill opacity=1 ] (76.3,110) .. controls (76.3,107.96) and (77.96,106.3) .. (80,106.3) .. controls (82.04,106.3) and (83.7,107.96) .. (83.7,110) .. controls (83.7,112.04) and (82.04,113.7) .. (80,113.7) .. controls (77.96,113.7) and (76.3,112.04) .. (76.3,110) -- cycle ;
\draw  [fill={rgb, 255:red, 0; green, 0; blue, 0 }  ,fill opacity=1 ] (236.3,110) .. controls (236.3,107.96) and (237.96,106.3) .. (240,106.3) .. controls (242.04,106.3) and (243.7,107.96) .. (243.7,110) .. controls (243.7,112.04) and (242.04,113.7) .. (240,113.7) .. controls (237.96,113.7) and (236.3,112.04) .. (236.3,110) -- cycle ;
\draw  [fill={rgb, 255:red, 0; green, 0; blue, 0 }  ,fill opacity=1 ] (76.3,170) .. controls (76.3,167.96) and (77.96,166.3) .. (80,166.3) .. controls (82.04,166.3) and (83.7,167.96) .. (83.7,170) .. controls (83.7,172.04) and (82.04,173.7) .. (80,173.7) .. controls (77.96,173.7) and (76.3,172.04) .. (76.3,170) -- cycle ;
\draw  [fill={rgb, 255:red, 0; green, 0; blue, 0 }  ,fill opacity=1 ] (236.3,170) .. controls (236.3,167.96) and (237.96,166.3) .. (240,166.3) .. controls (242.04,166.3) and (243.7,167.96) .. (243.7,170) .. controls (243.7,172.04) and (242.04,173.7) .. (240,173.7) .. controls (237.96,173.7) and (236.3,172.04) .. (236.3,170) -- cycle ;
\draw  [fill={rgb, 255:red, 0; green, 0; blue, 0 }  ,fill opacity=1 ] (136.3,240) .. controls (136.3,237.96) and (137.96,236.3) .. (140,236.3) .. controls (142.04,236.3) and (143.7,237.96) .. (143.7,240) .. controls (143.7,242.04) and (142.04,243.7) .. (140,243.7) .. controls (137.96,243.7) and (136.3,242.04) .. (136.3,240) -- cycle ;
\draw  [fill={rgb, 255:red, 0; green, 0; blue, 0 }  ,fill opacity=1 ] (76.3,240) .. controls (76.3,237.96) and (77.96,236.3) .. (80,236.3) .. controls (82.04,236.3) and (83.7,237.96) .. (83.7,240) .. controls (83.7,242.04) and (82.04,243.7) .. (80,243.7) .. controls (77.96,243.7) and (76.3,242.04) .. (76.3,240) -- cycle ;
\draw  [fill={rgb, 255:red, 0; green, 0; blue, 0 }  ,fill opacity=1 ] (436.3,110) .. controls (436.3,107.96) and (437.96,106.3) .. (440,106.3) .. controls (442.04,106.3) and (443.7,107.96) .. (443.7,110) .. controls (443.7,112.04) and (442.04,113.7) .. (440,113.7) .. controls (437.96,113.7) and (436.3,112.04) .. (436.3,110) -- cycle ;
\draw  [fill={rgb, 255:red, 0; green, 0; blue, 0 }  ,fill opacity=1 ] (496.3,110) .. controls (496.3,107.96) and (497.96,106.3) .. (500,106.3) .. controls (502.04,106.3) and (503.7,107.96) .. (503.7,110) .. controls (503.7,112.04) and (502.04,113.7) .. (500,113.7) .. controls (497.96,113.7) and (496.3,112.04) .. (496.3,110) -- cycle ;
\draw  [fill={rgb, 255:red, 0; green, 0; blue, 0 }  ,fill opacity=1 ] (596.3,110) .. controls (596.3,107.96) and (597.96,106.3) .. (600,106.3) .. controls (602.04,106.3) and (603.7,107.96) .. (603.7,110) .. controls (603.7,112.04) and (602.04,113.7) .. (600,113.7) .. controls (597.96,113.7) and (596.3,112.04) .. (596.3,110) -- cycle ;
\draw  [fill={rgb, 255:red, 0; green, 0; blue, 0 }  ,fill opacity=1 ] (436.3,240) .. controls (436.3,237.96) and (437.96,236.3) .. (440,236.3) .. controls (442.04,236.3) and (443.7,237.96) .. (443.7,240) .. controls (443.7,242.04) and (442.04,243.7) .. (440,243.7) .. controls (437.96,243.7) and (436.3,242.04) .. (436.3,240) -- cycle ;
\draw  [fill={rgb, 255:red, 0; green, 0; blue, 0 }  ,fill opacity=1 ] (496.3,170) .. controls (496.3,167.96) and (497.96,166.3) .. (500,166.3) .. controls (502.04,166.3) and (503.7,167.96) .. (503.7,170) .. controls (503.7,172.04) and (502.04,173.7) .. (500,173.7) .. controls (497.96,173.7) and (496.3,172.04) .. (496.3,170) -- cycle ;
\draw  [fill={rgb, 255:red, 0; green, 0; blue, 0 }  ,fill opacity=1 ] (496.3,240) .. controls (496.3,237.96) and (497.96,236.3) .. (500,236.3) .. controls (502.04,236.3) and (503.7,237.96) .. (503.7,240) .. controls (503.7,242.04) and (502.04,243.7) .. (500,243.7) .. controls (497.96,243.7) and (496.3,242.04) .. (496.3,240) -- cycle ;
\draw  [fill={rgb, 255:red, 0; green, 0; blue, 0 }  ,fill opacity=1 ] (596.3,170) .. controls (596.3,167.96) and (597.96,166.3) .. (600,166.3) .. controls (602.04,166.3) and (603.7,167.96) .. (603.7,170) .. controls (603.7,172.04) and (602.04,173.7) .. (600,173.7) .. controls (597.96,173.7) and (596.3,172.04) .. (596.3,170) -- cycle ;
\draw  [fill={rgb, 255:red, 0; green, 0; blue, 0 }  ,fill opacity=1 ] (136.3,170) .. controls (136.3,167.96) and (137.96,166.3) .. (140,166.3) .. controls (142.04,166.3) and (143.7,167.96) .. (143.7,170) .. controls (143.7,172.04) and (142.04,173.7) .. (140,173.7) .. controls (137.96,173.7) and (136.3,172.04) .. (136.3,170) -- cycle ;
\draw    (320,160) -- (388,160) ;
\draw [shift={(390,160)}, rotate = 180] [color={rgb, 255:red, 0; green, 0; blue, 0 }  ][line width=0.75]    (10.93,-3.29) .. controls (6.95,-1.4) and (3.31,-0.3) .. (0,0) .. controls (3.31,0.3) and (6.95,1.4) .. (10.93,3.29)   ;
\draw [shift={(320,160)}, rotate = 180] [color={rgb, 255:red, 0; green, 0; blue, 0 }  ][line width=0.75]    (0,5.59) -- (0,-5.59)   ;

\draw (149,130.4) node [anchor=north west][inner sep=0.75pt]    {$R_{\alpha }$};
\draw (99,192.4) node [anchor=north west][inner sep=0.75pt]    {$R_{\beta }$};
\draw (537,130.4) node [anchor=north west][inner sep=0.75pt]    {$R'_{\alpha }$};
\draw (461,172.4) node [anchor=north west][inner sep=0.75pt]    {$R'_{\beta }$};
\draw (341,126.4) node [anchor=north west][inner sep=0.75pt]    {$\Pi_{\alpha\beta}^{\mathrm{lb}}$};
\draw (49,82.4) node [anchor=north west][inner sep=0.75pt]    {$\Phi _{4\alpha }$};
\draw (239,82.4) node [anchor=north west][inner sep=0.75pt]    {$\Phi _{3\alpha }$};
\draw (240,169.7) node [anchor=north west][inner sep=0.75pt]    {$\Phi _{2\alpha }$};
\draw (51,240.4) node [anchor=north west][inner sep=0.75pt]    {$\Phi _{1\beta }$};
\draw (142,173.4) node [anchor=north west][inner sep=0.75pt]    {$\Phi _{3\beta }$};
\draw (142,243.4) node [anchor=north west][inner sep=0.75pt]    {$\Phi _{2\beta }$};
\draw (18,147.4) node [anchor=north west][inner sep=0.75pt]    {$ \begin{array}{l}
\Phi _{1\alpha } =\\
=\Phi _{4\beta }
\end{array}$};
\draw (411,82.4) node [anchor=north west][inner sep=0.75pt]    {$\Phi '_{4\beta }$};
\draw (502,173.4) node [anchor=north west][inner sep=0.75pt]    {$\Phi '_{1\alpha }$};
\draw (462,82.4) node [anchor=north west][inner sep=0.75pt]    {$\Phi '_{3\beta } =\Phi '_{4\alpha }$};
\draw (602,173.4) node [anchor=north west][inner sep=0.75pt]    {$\Phi '_{2\alpha }$};
\draw (601,82.4) node [anchor=north west][inner sep=0.75pt]    {$\Phi '_{3\alpha }$};
\draw (502,239.7) node [anchor=north west][inner sep=0.75pt]    {$\Phi '_{2\beta }$};
\draw (411,240.4) node [anchor=north west][inner sep=0.75pt]    {$\Phi '_{1\beta }$};

\end{tikzpicture}
\nopagebreak

$$
\Phi'_{3\beta}=\Phi'_{4\alpha}=F_\alpha\circ\Phi_{3\beta},\,\Phi'_{4\beta}=\Phi_{4\alpha},\,\Phi'_{1\alpha}=\Phi_{3\beta},\quad\text{while}\quad\Phi'_{i\xi}=\Phi_{i\xi},\,i\xi\not\in\{3\beta,4\alpha,4\beta,1\alpha\}.
$$

\subsubsection*{Case $\Pi=\Pi_{\alpha\beta}^{\mathrm{le}}$, $\alpha,\beta\in\A$:}
\nopagebreak

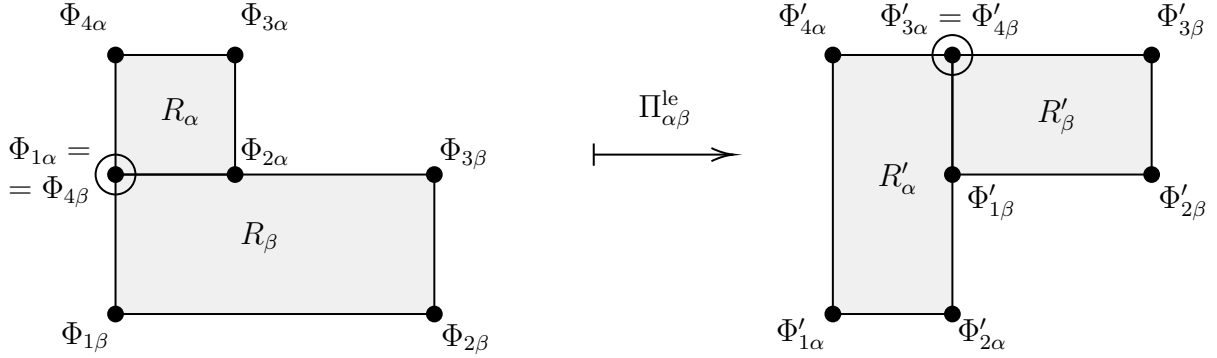
\begin{tikzpicture}[x=0.75pt,y=0.75pt,yscale=-1,xscale=1]

\draw  [fill={rgb, 255:red, 0; green, 0; blue, 0 }  ,fill opacity=0.06 ] (80,110) -- (140,110) -- (140,170) -- (80,170) -- cycle ;
\draw  [fill={rgb, 255:red, 0; green, 0; blue, 0 }  ,fill opacity=0.06 ] (80,170) -- (240,170) -- (240,240) -- (80,240) -- cycle ;
\draw  [fill={rgb, 255:red, 0; green, 0; blue, 0 }  ,fill opacity=0.06 ] (440,110) -- (500,110) -- (500,240) -- (440,240) -- cycle ;
\draw  [fill={rgb, 255:red, 0; green, 0; blue, 0 }  ,fill opacity=0.06 ] (500,110) -- (600,110) -- (600,170) -- (500,170) -- cycle ;
\draw   (70,170) .. controls (70,164.48) and (74.48,160) .. (80,160) .. controls (85.52,160) and (90,164.48) .. (90,170) .. controls (90,175.52) and (85.52,180) .. (80,180) .. controls (74.48,180) and (70,175.52) .. (70,170) -- cycle ;
\draw   (490,110) .. controls (490,104.48) and (494.48,100) .. (500,100) .. controls (505.52,100) and (510,104.48) .. (510,110) .. controls (510,115.52) and (505.52,120) .. (500,120) .. controls (494.48,120) and (490,115.52) .. (490,110) -- cycle ;
\draw  [fill={rgb, 255:red, 0; green, 0; blue, 0 }  ,fill opacity=1 ] (76.3,110) .. controls (76.3,107.96) and (77.96,106.3) .. (80,106.3) .. controls (82.04,106.3) and (83.7,107.96) .. (83.7,110) .. controls (83.7,112.04) and (82.04,113.7) .. (80,113.7) .. controls (77.96,113.7) and (76.3,112.04) .. (76.3,110) -- cycle ;
\draw  [fill={rgb, 255:red, 0; green, 0; blue, 0 }  ,fill opacity=1 ] (136.3,110) .. controls (136.3,107.96) and (137.96,106.3) .. (140,106.3) .. controls (142.04,106.3) and (143.7,107.96) .. (143.7,110) .. controls (143.7,112.04) and (142.04,113.7) .. (140,113.7) .. controls (137.96,113.7) and (136.3,112.04) .. (136.3,110) -- cycle ;
\draw  [fill={rgb, 255:red, 0; green, 0; blue, 0 }  ,fill opacity=1 ] (76.3,170) .. controls (76.3,167.96) and (77.96,166.3) .. (80,166.3) .. controls (82.04,166.3) and (83.7,167.96) .. (83.7,170) .. controls (83.7,172.04) and (82.04,173.7) .. (80,173.7) .. controls (77.96,173.7) and (76.3,172.04) .. (76.3,170) -- cycle ;
\draw  [fill={rgb, 255:red, 0; green, 0; blue, 0 }  ,fill opacity=1 ] (236.3,170) .. controls (236.3,167.96) and (237.96,166.3) .. (240,166.3) .. controls (242.04,166.3) and (243.7,167.96) .. (243.7,170) .. controls (243.7,172.04) and (242.04,173.7) .. (240,173.7) .. controls (237.96,173.7) and (236.3,172.04) .. (236.3,170) -- cycle ;
\draw  [fill={rgb, 255:red, 0; green, 0; blue, 0 }  ,fill opacity=1 ] (236.3,240) .. controls (236.3,237.96) and (237.96,236.3) .. (240,236.3) .. controls (242.04,236.3) and (243.7,237.96) .. (243.7,240) .. controls (243.7,242.04) and (242.04,243.7) .. (240,243.7) .. controls (237.96,243.7) and (236.3,242.04) .. (236.3,240) -- cycle ;
\draw  [fill={rgb, 255:red, 0; green, 0; blue, 0 }  ,fill opacity=1 ] (76.3,240) .. controls (76.3,237.96) and (77.96,236.3) .. (80,236.3) .. controls (82.04,236.3) and (83.7,237.96) .. (83.7,240) .. controls (83.7,242.04) and (82.04,243.7) .. (80,243.7) .. controls (77.96,243.7) and (76.3,242.04) .. (76.3,240) -- cycle ;
\draw  [fill={rgb, 255:red, 0; green, 0; blue, 0 }  ,fill opacity=1 ] (436.3,110) .. controls (436.3,107.96) and (437.96,106.3) .. (440,106.3) .. controls (442.04,106.3) and (443.7,107.96) .. (443.7,110) .. controls (443.7,112.04) and (442.04,113.7) .. (440,113.7) .. controls (437.96,113.7) and (436.3,112.04) .. (436.3,110) -- cycle ;
\draw  [fill={rgb, 255:red, 0; green, 0; blue, 0 }  ,fill opacity=1 ] (496.3,110) .. controls (496.3,107.96) and (497.96,106.3) .. (500,106.3) .. controls (502.04,106.3) and (503.7,107.96) .. (503.7,110) .. controls (503.7,112.04) and (502.04,113.7) .. (500,113.7) .. controls (497.96,113.7) and (496.3,112.04) .. (496.3,110) -- cycle ;
\draw  [fill={rgb, 255:red, 0; green, 0; blue, 0 }  ,fill opacity=1 ] (596.3,110) .. controls (596.3,107.96) and (597.96,106.3) .. (600,106.3) .. controls (602.04,106.3) and (603.7,107.96) .. (603.7,110) .. controls (603.7,112.04) and (602.04,113.7) .. (600,113.7) .. controls (597.96,113.7) and (596.3,112.04) .. (596.3,110) -- cycle ;
\draw  [fill={rgb, 255:red, 0; green, 0; blue, 0 }  ,fill opacity=1 ] (436.3,240) .. controls (436.3,237.96) and (437.96,236.3) .. (440,236.3) .. controls (442.04,236.3) and (443.7,237.96) .. (443.7,240) .. controls (443.7,242.04) and (442.04,243.7) .. (440,243.7) .. controls (437.96,243.7) and (436.3,242.04) .. (436.3,240) -- cycle ;
\draw  [fill={rgb, 255:red, 0; green, 0; blue, 0 }  ,fill opacity=1 ] (496.3,170) .. controls (496.3,167.96) and (497.96,166.3) .. (500,166.3) .. controls (502.04,166.3) and (503.7,167.96) .. (503.7,170) .. controls (503.7,172.04) and (502.04,173.7) .. (500,173.7) .. controls (497.96,173.7) and (496.3,172.04) .. (496.3,170) -- cycle ;
\draw  [fill={rgb, 255:red, 0; green, 0; blue, 0 }  ,fill opacity=1 ] (496.3,240) .. controls (496.3,237.96) and (497.96,236.3) .. (500,236.3) .. controls (502.04,236.3) and (503.7,237.96) .. (503.7,240) .. controls (503.7,242.04) and (502.04,243.7) .. (500,243.7) .. controls (497.96,243.7) and (496.3,242.04) .. (496.3,240) -- cycle ;
\draw  [fill={rgb, 255:red, 0; green, 0; blue, 0 }  ,fill opacity=1 ] (596.3,170) .. controls (596.3,167.96) and (597.96,166.3) .. (600,166.3) .. controls (602.04,166.3) and (603.7,167.96) .. (603.7,170) .. controls (603.7,172.04) and (602.04,173.7) .. (600,173.7) .. controls (597.96,173.7) and (596.3,172.04) .. (596.3,170) -- cycle ;
\draw  [fill={rgb, 255:red, 0; green, 0; blue, 0 }  ,fill opacity=1 ] (136.3,170) .. controls (136.3,167.96) and (137.96,166.3) .. (140,166.3) .. controls (142.04,166.3) and (143.7,167.96) .. (143.7,170) .. controls (143.7,172.04) and (142.04,173.7) .. (140,173.7) .. controls (137.96,173.7) and (136.3,172.04) .. (136.3,170) -- cycle ;
\draw    (320,160) -- (388,160) ;
\draw [shift={(390,160)}, rotate = 180] [color={rgb, 255:red, 0; green, 0; blue, 0 }  ][line width=0.75]    (10.93,-3.29) .. controls (6.95,-1.4) and (3.31,-0.3) .. (0,0) .. controls (3.31,0.3) and (6.95,1.4) .. (10.93,3.29)   ;
\draw [shift={(320,160)}, rotate = 180] [color={rgb, 255:red, 0; green, 0; blue, 0 }  ][line width=0.75]    (0,5.59) -- (0,-5.59)   ;

\draw (101,130.4) node [anchor=north west][inner sep=0.75pt]    {$R_{\alpha }$};
\draw (141,192.4) node [anchor=north west][inner sep=0.75pt]    {$R_{\beta }$};
\draw (461,162.4) node [anchor=north west][inner sep=0.75pt]    {$R'_{\alpha }$};
\draw (541,130.4) node [anchor=north west][inner sep=0.75pt]    {$R'_{\beta }$};
\draw (341,126.4) node [anchor=north west][inner sep=0.75pt]    {$\Pi_{\alpha\beta}^{\mathrm{le}}$};
\draw (50,82.4) node [anchor=north west][inner sep=0.75pt]    {$\Phi _{4\alpha }$};
\draw (141,82.4) node [anchor=north west][inner sep=0.75pt]    {$\Phi _{3\alpha }$};
\draw (141,150.4) node [anchor=north west][inner sep=0.75pt]    {$\Phi _{2\alpha }$};
\draw (51,242.4) node [anchor=north west][inner sep=0.75pt]    {$\Phi _{1\beta }$};
\draw (241,150.4) node [anchor=north west][inner sep=0.75pt]    {$\Phi _{3\beta }$};
\draw (242,243.4) node [anchor=north west][inner sep=0.75pt]    {$\Phi _{2\beta }$};
\draw (18,147.4) node [anchor=north west][inner sep=0.75pt]    {$ \begin{array}{l}
\Phi _{1\alpha } =\\
=\Phi _{4\beta }
\end{array}$};
\draw (411,82.4) node [anchor=north west][inner sep=0.75pt]    {$\Phi '_{4\alpha }$};
\draw (410,240.4) node [anchor=north west][inner sep=0.75pt]    {$\Phi '_{1\alpha }$};
\draw (462,82.4) node [anchor=north west][inner sep=0.75pt]    {$\Phi '_{3\alpha } =\Phi '_{4\beta }$};
\draw (501,240.4) node [anchor=north west][inner sep=0.75pt]    {$\Phi '_{2\alpha }$};
\draw (601,82.4) node [anchor=north west][inner sep=0.75pt]    {$\Phi '_{3\beta }$};
\draw (602,173.4) node [anchor=north west][inner sep=0.75pt]    {$\Phi '_{2\beta }$};
\draw (505.7,173.4) node [anchor=north west][inner sep=0.75pt]    {$\Phi '_{1\beta }$};

\end{tikzpicture}

\nopagebreak

$$
\Phi'_{1\beta}=\Phi'_{2\alpha}=F_\beta^{-1}\circ\Phi_{2\alpha},\,\Phi'_{1\alpha}=\Phi_{1\beta},\,\Phi'_{4\beta}=\Phi_{2\alpha},\quad\text{while}\quad\Phi'_{i\xi}=\Phi_{i\xi},\,i\xi\not\in\{1\beta,2\alpha,1\alpha,4\beta\}.
$$

The most important conclusion from the results obtained in this subsection is that when we visualize the nodal functions as written at the corners of rectangles on the toroidal patchwork surface, then---in each of the four cases of applying elementary induction steps to such a surface and to an array of such functions---the corresponding nodal functions are preserved at all nodes, which were preserved by the given step (that is, actually at all nodes except for the one that disappeared and one newly born), and exactly two functions from the array change their names. Namely: in the case $\Pi=\Pi_{\alpha\beta}^{\mathrm{rb}}$, the nodal functions $\Phi_{3\alpha}$ and $\Phi_{4\beta}$ change their names to $\Phi'_{3\beta}$ and $\Phi'_{2\alpha}$ respectively; in the case $\Pi=\Pi_{\alpha\beta}^{\mathrm{re}}$, the functions $\Phi_{2\beta}$ and $\Phi_{1\alpha}$ change their names to $\Phi'_{2\alpha}$ and $\Phi'_{3\beta}$, respectively; in the case $\Pi=\Pi_{\alpha\beta}^{\mathrm{lb}}$, the functions $\Phi_{4\alpha}$ and $\Phi_{3\beta}$ change their names to $\Phi'_{4\beta}$ and $\Phi'_{1\alpha}$ respectively; and in the case $\Pi=\Pi_{\alpha\beta}^{\mathrm{le}}$, the functions $\Phi_{1\beta}$ and $\Phi_{2\alpha}$ change their names to $\Phi'_{1\alpha}$ and $\Phi'_{4\beta}$ respectively. However, all these nodal functions remain unchanged and are written at the very same corners on the patchwork surface of the natural extension at which corners they were written before application of the induction step $\Pi$.

\subsection{Proof of Theorem~\ref{theorem:1}}
\label{subsect:proof}

Let us recall once again that, according to Subsection~\ref{subsect:Phi}, commuting collections $\Fect$ are in one-to-one correspondence with arrays of nodal functions $\Phict$, and we will conduct the proof of the theorem in terms of nodal functions. In Subsection~\ref{subsect:induction_Phi} we described how an array of nodal functions changes under the action of an elementary induction step $\Pi$ (in all four variants of such a step), and in Subsection~\ref{subsect:dual_surface} we described how such an array changes under the action of the duality involution $\Inv$. Now we simply collect all the information obtained above and show that the equality
$$
\Pi^*\Inv\Pi\Phict=\Inv\Phict
$$
indeed holds for all arrays $\Phict$ such that the elementary induction steps $\Pi$ and $\Pi^*$ are allowed for the corresponding schemes---that is, the step $\Pi$ is allowed for the scheme $\sigma$ corresponding to the array $\Phict$, while the step $\Pi^*$ is allowed for the scheme $\sigma'^*=\Inv\Pi\sigma$ corresponding to the array $\Phict'^*=\Inv\Pi\Phict$. This will prove the theorem according to the explanation in Subsection~\ref{subsect:theorem1}.

Let us again consider in detail the case of the elementary induction step $\Pi=\Pi_{\alpha\beta}^{\mathrm{rb}}$; the cases of the other three steps $\Pi_{\alpha\beta}^{\mathrm{re}}$, $\Pi_{\alpha\beta}^{\mathrm{lb}}$ and $\Pi_{\alpha\beta}^{\mathrm{le}}$, $\alpha,\beta\in\A$, are considered similarly. For visualization of our reasoning, the figure below shows the transformation of the corresponding part of the patchwork surface of the natural extension $[(\sigma,\uect), (\sigma^*,\yect)]$, which we visualized in the previous subsection. Recall that under the duality involution this surface is simply reflected symmetrically across the identity line $y=u$, while the nodal functions are transposed as matrices from the space $\psl$.

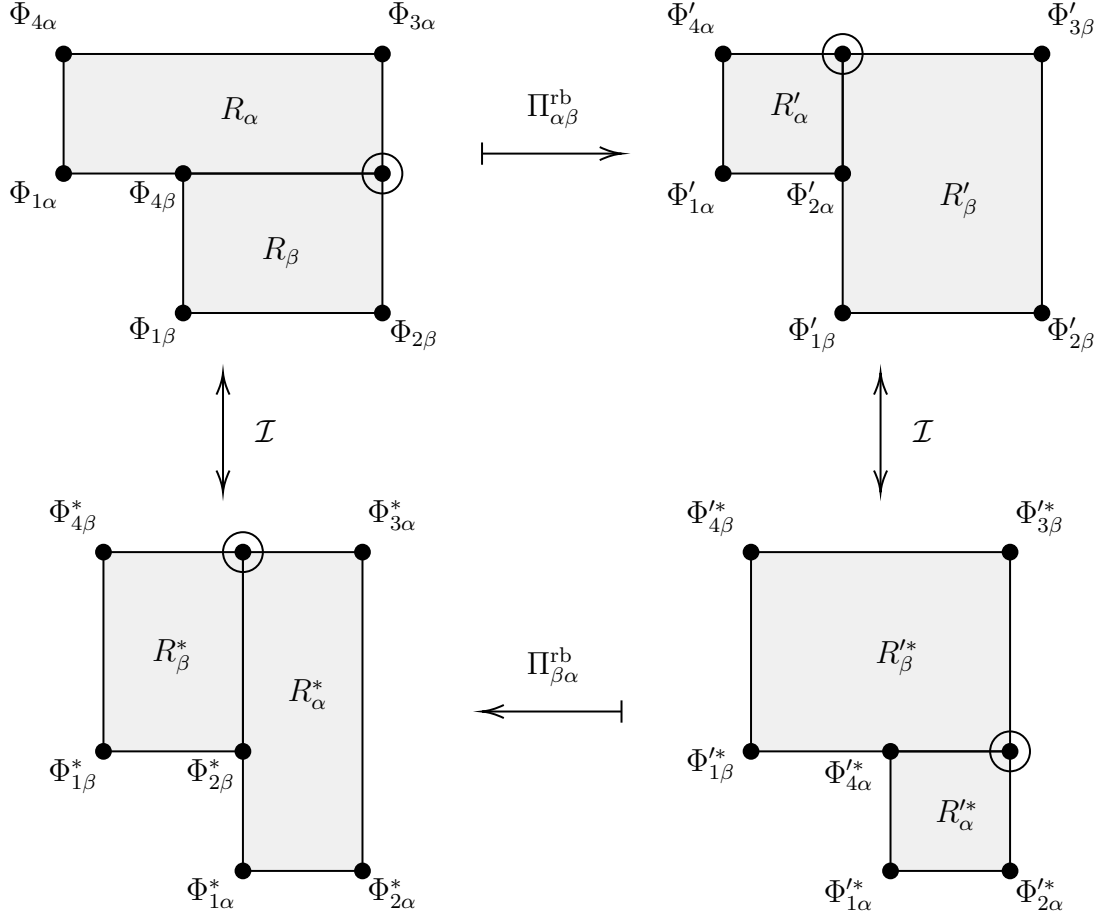
\vspace{20pt}

\begin{tikzpicture}[x=0.75pt,y=0.75pt,yscale=-1,xscale=1]

\draw  [fill={rgb, 255:red, 0; green, 0; blue, 0 }  ,fill opacity=0.06 ] (100,40) -- (260,40) -- (260,100) -- (100,100) -- cycle ;
\draw  [fill={rgb, 255:red, 0; green, 0; blue, 0 }  ,fill opacity=0.06 ] (160,100) -- (260,100) -- (260,170) -- (160,170) -- cycle ;
\draw  [fill={rgb, 255:red, 0; green, 0; blue, 0 }  ,fill opacity=0.06 ] (431,40) -- (491,40) -- (491,100) -- (431,100) -- cycle ;
\draw  [fill={rgb, 255:red, 0; green, 0; blue, 0 }  ,fill opacity=0.06 ] (491,40) -- (591,40) -- (591,170) -- (491,170) -- cycle ;
\draw   (250,100) .. controls (250,94.48) and (254.48,90) .. (260,90) .. controls (265.52,90) and (270,94.48) .. (270,100) .. controls (270,105.52) and (265.52,110) .. (260,110) .. controls (254.48,110) and (250,105.52) .. (250,100) -- cycle ;
\draw   (481,40) .. controls (481,34.48) and (485.48,30) .. (491,30) .. controls (496.52,30) and (501,34.48) .. (501,40) .. controls (501,45.52) and (496.52,50) .. (491,50) .. controls (485.48,50) and (481,45.52) .. (481,40) -- cycle ;
\draw  [fill={rgb, 255:red, 0; green, 0; blue, 0 }  ,fill opacity=1 ] (96.3,40) .. controls (96.3,37.96) and (97.96,36.3) .. (100,36.3) .. controls (102.04,36.3) and (103.7,37.96) .. (103.7,40) .. controls (103.7,42.04) and (102.04,43.7) .. (100,43.7) .. controls (97.96,43.7) and (96.3,42.04) .. (96.3,40) -- cycle ;
\draw  [fill={rgb, 255:red, 0; green, 0; blue, 0 }  ,fill opacity=1 ] (256.3,40) .. controls (256.3,37.96) and (257.96,36.3) .. (260,36.3) .. controls (262.04,36.3) and (263.7,37.96) .. (263.7,40) .. controls (263.7,42.04) and (262.04,43.7) .. (260,43.7) .. controls (257.96,43.7) and (256.3,42.04) .. (256.3,40) -- cycle ;
\draw  [fill={rgb, 255:red, 0; green, 0; blue, 0 }  ,fill opacity=1 ] (96.3,100) .. controls (96.3,97.96) and (97.96,96.3) .. (100,96.3) .. controls (102.04,96.3) and (103.7,97.96) .. (103.7,100) .. controls (103.7,102.04) and (102.04,103.7) .. (100,103.7) .. controls (97.96,103.7) and (96.3,102.04) .. (96.3,100) -- cycle ;
\draw  [fill={rgb, 255:red, 0; green, 0; blue, 0 }  ,fill opacity=1 ] (256.3,100) .. controls (256.3,97.96) and (257.96,96.3) .. (260,96.3) .. controls (262.04,96.3) and (263.7,97.96) .. (263.7,100) .. controls (263.7,102.04) and (262.04,103.7) .. (260,103.7) .. controls (257.96,103.7) and (256.3,102.04) .. (256.3,100) -- cycle ;
\draw  [fill={rgb, 255:red, 0; green, 0; blue, 0 }  ,fill opacity=1 ] (156.3,170) .. controls (156.3,167.96) and (157.96,166.3) .. (160,166.3) .. controls (162.04,166.3) and (163.7,167.96) .. (163.7,170) .. controls (163.7,172.04) and (162.04,173.7) .. (160,173.7) .. controls (157.96,173.7) and (156.3,172.04) .. (156.3,170) -- cycle ;
\draw  [fill={rgb, 255:red, 0; green, 0; blue, 0 }  ,fill opacity=1 ] (256.3,170) .. controls (256.3,167.96) and (257.96,166.3) .. (260,166.3) .. controls (262.04,166.3) and (263.7,167.96) .. (263.7,170) .. controls (263.7,172.04) and (262.04,173.7) .. (260,173.7) .. controls (257.96,173.7) and (256.3,172.04) .. (256.3,170) -- cycle ;
\draw  [fill={rgb, 255:red, 0; green, 0; blue, 0 }  ,fill opacity=1 ] (427.3,40) .. controls (427.3,37.96) and (428.96,36.3) .. (431,36.3) .. controls (433.04,36.3) and (434.7,37.96) .. (434.7,40) .. controls (434.7,42.04) and (433.04,43.7) .. (431,43.7) .. controls (428.96,43.7) and (427.3,42.04) .. (427.3,40) -- cycle ;
\draw  [fill={rgb, 255:red, 0; green, 0; blue, 0 }  ,fill opacity=1 ] (487.3,40) .. controls (487.3,37.96) and (488.96,36.3) .. (491,36.3) .. controls (493.04,36.3) and (494.7,37.96) .. (494.7,40) .. controls (494.7,42.04) and (493.04,43.7) .. (491,43.7) .. controls (488.96,43.7) and (487.3,42.04) .. (487.3,40) -- cycle ;
\draw  [fill={rgb, 255:red, 0; green, 0; blue, 0 }  ,fill opacity=1 ] (587.3,40) .. controls (587.3,37.96) and (588.96,36.3) .. (591,36.3) .. controls (593.04,36.3) and (594.7,37.96) .. (594.7,40) .. controls (594.7,42.04) and (593.04,43.7) .. (591,43.7) .. controls (588.96,43.7) and (587.3,42.04) .. (587.3,40) -- cycle ;
\draw  [fill={rgb, 255:red, 0; green, 0; blue, 0 }  ,fill opacity=1 ] (427.3,100) .. controls (427.3,97.96) and (428.96,96.3) .. (431,96.3) .. controls (433.04,96.3) and (434.7,97.96) .. (434.7,100) .. controls (434.7,102.04) and (433.04,103.7) .. (431,103.7) .. controls (428.96,103.7) and (427.3,102.04) .. (427.3,100) -- cycle ;
\draw  [fill={rgb, 255:red, 0; green, 0; blue, 0 }  ,fill opacity=1 ] (487.3,100) .. controls (487.3,97.96) and (488.96,96.3) .. (491,96.3) .. controls (493.04,96.3) and (494.7,97.96) .. (494.7,100) .. controls (494.7,102.04) and (493.04,103.7) .. (491,103.7) .. controls (488.96,103.7) and (487.3,102.04) .. (487.3,100) -- cycle ;
\draw  [fill={rgb, 255:red, 0; green, 0; blue, 0 }  ,fill opacity=1 ] (487.3,170) .. controls (487.3,167.96) and (488.96,166.3) .. (491,166.3) .. controls (493.04,166.3) and (494.7,167.96) .. (494.7,170) .. controls (494.7,172.04) and (493.04,173.7) .. (491,173.7) .. controls (488.96,173.7) and (487.3,172.04) .. (487.3,170) -- cycle ;
\draw  [fill={rgb, 255:red, 0; green, 0; blue, 0 }  ,fill opacity=1 ] (587.3,170) .. controls (587.3,167.96) and (588.96,166.3) .. (591,166.3) .. controls (593.04,166.3) and (594.7,167.96) .. (594.7,170) .. controls (594.7,172.04) and (593.04,173.7) .. (591,173.7) .. controls (588.96,173.7) and (587.3,172.04) .. (587.3,170) -- cycle ;
\draw  [fill={rgb, 255:red, 0; green, 0; blue, 0 }  ,fill opacity=1 ] (156.3,100) .. controls (156.3,97.96) and (157.96,96.3) .. (160,96.3) .. controls (162.04,96.3) and (163.7,97.96) .. (163.7,100) .. controls (163.7,102.04) and (162.04,103.7) .. (160,103.7) .. controls (157.96,103.7) and (156.3,102.04) .. (156.3,100) -- cycle ;
\draw    (310,90) -- (378,90) ;
\draw [shift={(380,90)}, rotate = 180] [color={rgb, 255:red, 0; green, 0; blue, 0 }  ][line width=0.75]    (10.93,-3.29) .. controls (6.95,-1.4) and (3.31,-0.3) .. (0,0) .. controls (3.31,0.3) and (6.95,1.4) .. (10.93,3.29)   ;
\draw [shift={(310,90)}, rotate = 180] [color={rgb, 255:red, 0; green, 0; blue, 0 }  ][line width=0.75]    (0,5.59) -- (0,-5.59)   ;
\draw  [fill={rgb, 255:red, 0; green, 0; blue, 0 }  ,fill opacity=0.06 ] (445,290) -- (575,290) -- (575,390) -- (445,390) -- cycle ;
\draw  [fill={rgb, 255:red, 0; green, 0; blue, 0 }  ,fill opacity=0.06 ] (515,390) -- (575,390) -- (575,450) -- (515,450) -- cycle ;
\draw  [fill={rgb, 255:red, 0; green, 0; blue, 0 }  ,fill opacity=1 ] (571.3,290) .. controls (571.3,287.96) and (572.96,286.3) .. (575,286.3) .. controls (577.04,286.3) and (578.7,287.96) .. (578.7,290) .. controls (578.7,292.04) and (577.04,293.7) .. (575,293.7) .. controls (572.96,293.7) and (571.3,292.04) .. (571.3,290) -- cycle ;
\draw  [fill={rgb, 255:red, 0; green, 0; blue, 0 }  ,fill opacity=1 ] (441.3,290) .. controls (441.3,287.96) and (442.96,286.3) .. (445,286.3) .. controls (447.04,286.3) and (448.7,287.96) .. (448.7,290) .. controls (448.7,292.04) and (447.04,293.7) .. (445,293.7) .. controls (442.96,293.7) and (441.3,292.04) .. (441.3,290) -- cycle ;
\draw  [fill={rgb, 255:red, 0; green, 0; blue, 0 }  ,fill opacity=1 ] (571.3,390) .. controls (571.3,387.96) and (572.96,386.3) .. (575,386.3) .. controls (577.04,386.3) and (578.7,387.96) .. (578.7,390) .. controls (578.7,392.04) and (577.04,393.7) .. (575,393.7) .. controls (572.96,393.7) and (571.3,392.04) .. (571.3,390) -- cycle ;
\draw  [fill={rgb, 255:red, 0; green, 0; blue, 0 }  ,fill opacity=1 ] (441.3,390) .. controls (441.3,387.96) and (442.96,386.3) .. (445,386.3) .. controls (447.04,386.3) and (448.7,387.96) .. (448.7,390) .. controls (448.7,392.04) and (447.04,393.7) .. (445,393.7) .. controls (442.96,393.7) and (441.3,392.04) .. (441.3,390) -- cycle ;
\draw  [fill={rgb, 255:red, 0; green, 0; blue, 0 }  ,fill opacity=1 ] (511.3,390) .. controls (511.3,387.96) and (512.96,386.3) .. (515,386.3) .. controls (517.04,386.3) and (518.7,387.96) .. (518.7,390) .. controls (518.7,392.04) and (517.04,393.7) .. (515,393.7) .. controls (512.96,393.7) and (511.3,392.04) .. (511.3,390) -- cycle ;
\draw  [fill={rgb, 255:red, 0; green, 0; blue, 0 }  ,fill opacity=1 ] (511.3,450) .. controls (511.3,447.96) and (512.96,446.3) .. (515,446.3) .. controls (517.04,446.3) and (518.7,447.96) .. (518.7,450) .. controls (518.7,452.04) and (517.04,453.7) .. (515,453.7) .. controls (512.96,453.7) and (511.3,452.04) .. (511.3,450) -- cycle ;
\draw  [fill={rgb, 255:red, 0; green, 0; blue, 0 }  ,fill opacity=1 ] (571.3,450) .. controls (571.3,447.96) and (572.96,446.3) .. (575,446.3) .. controls (577.04,446.3) and (578.7,447.96) .. (578.7,450) .. controls (578.7,452.04) and (577.04,453.7) .. (575,453.7) .. controls (572.96,453.7) and (571.3,452.04) .. (571.3,450) -- cycle ;
\draw    (510,202) -- (510,258) ;
\draw [shift={(510,260)}, rotate = 270] [color={rgb, 255:red, 0; green, 0; blue, 0 }  ][line width=0.75]    (10.93,-3.29) .. controls (6.95,-1.4) and (3.31,-0.3) .. (0,0) .. controls (3.31,0.3) and (6.95,1.4) .. (10.93,3.29)   ;
\draw [shift={(510,200)}, rotate = 90] [color={rgb, 255:red, 0; green, 0; blue, 0 }  ][line width=0.75]    (10.93,-3.29) .. controls (6.95,-1.4) and (3.31,-0.3) .. (0,0) .. controls (3.31,0.3) and (6.95,1.4) .. (10.93,3.29)   ;
\draw  [fill={rgb, 255:red, 0; green, 0; blue, 0 }  ,fill opacity=0.06 ] (120,290) -- (190,290) -- (190,390) -- (120,390) -- cycle ;
\draw  [fill={rgb, 255:red, 0; green, 0; blue, 0 }  ,fill opacity=0.06 ] (190,290) -- (250,290) -- (250,450) -- (190,450) -- cycle ;
\draw  [fill={rgb, 255:red, 0; green, 0; blue, 0 }  ,fill opacity=1 ] (246.3,290) .. controls (246.3,287.96) and (247.96,286.3) .. (250,286.3) .. controls (252.04,286.3) and (253.7,287.96) .. (253.7,290) .. controls (253.7,292.04) and (252.04,293.7) .. (250,293.7) .. controls (247.96,293.7) and (246.3,292.04) .. (246.3,290) -- cycle ;
\draw  [fill={rgb, 255:red, 0; green, 0; blue, 0 }  ,fill opacity=1 ] (116.3,290) .. controls (116.3,287.96) and (117.96,286.3) .. (120,286.3) .. controls (122.04,286.3) and (123.7,287.96) .. (123.7,290) .. controls (123.7,292.04) and (122.04,293.7) .. (120,293.7) .. controls (117.96,293.7) and (116.3,292.04) .. (116.3,290) -- cycle ;
\draw  [fill={rgb, 255:red, 0; green, 0; blue, 0 }  ,fill opacity=1 ] (186.3,290) .. controls (186.3,287.96) and (187.96,286.3) .. (190,286.3) .. controls (192.04,286.3) and (193.7,287.96) .. (193.7,290) .. controls (193.7,292.04) and (192.04,293.7) .. (190,293.7) .. controls (187.96,293.7) and (186.3,292.04) .. (186.3,290) -- cycle ;
\draw  [fill={rgb, 255:red, 0; green, 0; blue, 0 }  ,fill opacity=1 ] (116.3,390) .. controls (116.3,387.96) and (117.96,386.3) .. (120,386.3) .. controls (122.04,386.3) and (123.7,387.96) .. (123.7,390) .. controls (123.7,392.04) and (122.04,393.7) .. (120,393.7) .. controls (117.96,393.7) and (116.3,392.04) .. (116.3,390) -- cycle ;
\draw  [fill={rgb, 255:red, 0; green, 0; blue, 0 }  ,fill opacity=1 ] (186.3,390) .. controls (186.3,387.96) and (187.96,386.3) .. (190,386.3) .. controls (192.04,386.3) and (193.7,387.96) .. (193.7,390) .. controls (193.7,392.04) and (192.04,393.7) .. (190,393.7) .. controls (187.96,393.7) and (186.3,392.04) .. (186.3,390) -- cycle ;
\draw  [fill={rgb, 255:red, 0; green, 0; blue, 0 }  ,fill opacity=1 ] (186.3,450) .. controls (186.3,447.96) and (187.96,446.3) .. (190,446.3) .. controls (192.04,446.3) and (193.7,447.96) .. (193.7,450) .. controls (193.7,452.04) and (192.04,453.7) .. (190,453.7) .. controls (187.96,453.7) and (186.3,452.04) .. (186.3,450) -- cycle ;
\draw  [fill={rgb, 255:red, 0; green, 0; blue, 0 }  ,fill opacity=1 ] (246.3,450) .. controls (246.3,447.96) and (247.96,446.3) .. (250,446.3) .. controls (252.04,446.3) and (253.7,447.96) .. (253.7,450) .. controls (253.7,452.04) and (252.04,453.7) .. (250,453.7) .. controls (247.96,453.7) and (246.3,452.04) .. (246.3,450) -- cycle ;
\draw    (180,202) -- (180,258) ;
\draw [shift={(180,260)}, rotate = 270] [color={rgb, 255:red, 0; green, 0; blue, 0 }  ][line width=0.75]    (10.93,-3.29) .. controls (6.95,-1.4) and (3.31,-0.3) .. (0,0) .. controls (3.31,0.3) and (6.95,1.4) .. (10.93,3.29)   ;
\draw [shift={(180,200)}, rotate = 90] [color={rgb, 255:red, 0; green, 0; blue, 0 }  ][line width=0.75]    (10.93,-3.29) .. controls (6.95,-1.4) and (3.31,-0.3) .. (0,0) .. controls (3.31,0.3) and (6.95,1.4) .. (10.93,3.29)   ;
\draw   (565,390) .. controls (565,384.48) and (569.48,380) .. (575,380) .. controls (580.52,380) and (585,384.48) .. (585,390) .. controls (585,395.52) and (580.52,400) .. (575,400) .. controls (569.48,400) and (565,395.52) .. (565,390) -- cycle ;
\draw   (180,290) .. controls (180,284.48) and (184.48,280) .. (190,280) .. controls (195.52,280) and (200,284.48) .. (200,290) .. controls (200,295.52) and (195.52,300) .. (190,300) .. controls (184.48,300) and (180,295.52) .. (180,290) -- cycle ;
\draw    (312,370) -- (380,370) ;
\draw [shift={(380,370)}, rotate = 180] [color={rgb, 255:red, 0; green, 0; blue, 0 }  ][line width=0.75]    (0,5.59) -- (0,-5.59)   ;
\draw [shift={(310,370)}, rotate = 0] [color={rgb, 255:red, 0; green, 0; blue, 0 }  ][line width=0.75]    (10.93,-3.29) .. controls (6.95,-1.4) and (3.31,-0.3) .. (0,0) .. controls (3.31,0.3) and (6.95,1.4) .. (10.93,3.29)   ;

\draw (177,60.4) node [anchor=north west][inner sep=0.75pt]    {$R_{\alpha }$};
\draw (198,130.4) node [anchor=north west][inner sep=0.75pt]    {$R_{\beta }$};
\draw (452.99,57.57) node [anchor=north west][inner sep=0.75pt]    {$R'_{\alpha }$};
\draw (537.99,102.79) node [anchor=north west][inner sep=0.75pt]    {$R'_{\beta }$};
\draw (331,56.4) node [anchor=north west][inner sep=0.75pt]    {$\Pi _{\alpha \beta }^\mathrm{rb}$};
\draw (71,12.4) node [anchor=north west][inner sep=0.75pt]    {$\Phi _{4\alpha }$};
\draw (261,12.4) node [anchor=north west][inner sep=0.75pt]    {$\Phi _{3\alpha }$};
\draw (71,102.4) node [anchor=north west][inner sep=0.75pt]    {$\Phi _{1\alpha }$};
\draw (131,170.4) node [anchor=north west][inner sep=0.75pt]    {$\Phi _{1\beta }$};
\draw (131,102.4) node [anchor=north west][inner sep=0.75pt]    {$\Phi _{4\beta }$};
\draw (262,171.9) node [anchor=north west][inner sep=0.75pt]    {$\Phi _{2\beta }$};
\draw (401,12.4) node [anchor=north west][inner sep=0.75pt]    {$\Phi '_{4\alpha }$};
\draw (401,102.4) node [anchor=north west][inner sep=0.75pt]    {$\Phi '_{1\alpha }$};
\draw (461,102.4) node [anchor=north west][inner sep=0.75pt]    {$\Phi '_{2\alpha }$};
\draw (592,12.4) node [anchor=north west][inner sep=0.75pt]    {$\Phi '_{3\beta }$};
\draw (592,170.4) node [anchor=north west][inner sep=0.75pt]    {$\Phi '_{2\beta }$};
\draw (462,170.4) node [anchor=north west][inner sep=0.75pt]    {$\Phi '_{1\beta }$};
\draw (506,332.4) node [anchor=north west][inner sep=0.75pt]    {$R^{\prime *}_{\beta }$};
\draw (536,412.4) node [anchor=north west][inner sep=0.75pt]    {$R^{\prime *}_{\alpha }$};
\draw (411,262.4) node [anchor=north west][inner sep=0.75pt]    {$\Phi ^{\prime *}_{4\beta }$};
\draw (576,262.4) node [anchor=north west][inner sep=0.75pt]    {$\Phi ^{\prime *}_{3\beta }$};
\draw (411,386.4) node [anchor=north west][inner sep=0.75pt]    {$\Phi ^{\prime *}_{1\beta }$};
\draw (480,452.4) node [anchor=north west][inner sep=0.75pt]    {$\Phi ^{\prime *}_{1\alpha }$};
\draw (576,452.4) node [anchor=north west][inner sep=0.75pt]    {$\Phi ^{\prime *}_{2\alpha }$};
\draw (480,392.4) node [anchor=north west][inner sep=0.75pt]    {$\Phi ^{\prime *}_{4\alpha }$};
\draw (525,222.4) node [anchor=north west][inner sep=0.75pt]    {$\mathcal{I}$};
\draw (143,332.4) node [anchor=north west][inner sep=0.75pt]    {$R_{\beta }^{*}$};
\draw (211,352.4) node [anchor=north west][inner sep=0.75pt]    {$R_{\alpha }^{*}$};
\draw (91,262.4) node [anchor=north west][inner sep=0.75pt]    {$\Phi _{4\beta }^{*}$};
\draw (251,262.4) node [anchor=north west][inner sep=0.75pt]    {$\Phi _{3\alpha }^{*}$};
\draw (91,392.4) node [anchor=north west][inner sep=0.75pt]    {$\Phi _{1\beta }^{*}$};
\draw (160,452.4) node [anchor=north west][inner sep=0.75pt]    {$\Phi _{1\alpha }^{*}$};
\draw (251,452.4) node [anchor=north west][inner sep=0.75pt]    {$\Phi _{2\alpha }^{*}$};
\draw (160,392.4) node [anchor=north west][inner sep=0.75pt]    {$\Phi _{2\beta }^{*}$};
\draw (195,222.4) node [anchor=north west][inner sep=0.75pt]    {$\mathcal{I}$};
\draw (331,336.4) node [anchor=north west][inner sep=0.75pt]    {$\Pi _{\beta \alpha }^\mathrm{rb}$};

\end{tikzpicture}
\vspace{20pt}

We do not even need to follow the transformations of those nodal functions which are written at the corners meeting at the nodes marked with the small circles, because when three nodal functions are defined at the corners of some rectangle, the fourth is determined automatically: indeed, for any array $\Phict$ we have the relations $\Phi_{4\xi}\circ\Phi^{-1}_{1\xi}=\Phi_{3\xi}\circ\Phi^{-1}_{2\xi}$, $\xi\in\A$, according to Rule~3) from Subsection~\ref{subsect:Phi}. In fact, we could finish the proof here, because it has already been noted that when the induction step is applied, the nodal functions are preserved without change at all preserved nodes, while under the duality operation they are transposed, and everything is clear from the figure. But as a formality let us write down the transformations of these functions according to the formulas obtained in Subsections~\ref{subsect:induction_Phi} and \ref{subsect:dual_surface}.

Thus, according to Subsection~\ref{subsect:induction_Phi}, for the array $\Phict'=\Pi\Phict=\Pi_{\alpha\beta}^{\mathrm{rb}}\Phict$ we have the equalities: $\Phi'_{3\beta}=\Phi_{3\alpha}$, $\Phi'_{2\alpha}=\Phi_{4\beta}$, while $\Phi'_{i\xi}=\Phi_{i\xi}$, $i\xi\not\in\{4\beta,3\alpha,3\beta,2\alpha\}$.

From these equalities, according to Subsection~\ref{subsect:dual_surface}, for the array $\Phict'^*=\Inv\Phict'=\Inv\Pi\Phict$ we obtain the following equalities: $\Phi'^*_{1\xi}=(\Phi'_{1\xi})^\intercal=\Phi^\intercal_{1\xi}$ for all $\xi\in\A$; $\Phi'^*_{2\xi}=(\Phi'_{4\xi})^\intercal=\Phi^\intercal_{4\xi}$ for all $\xi\ne\beta$; $\Phi'^*_{3\beta}=(\Phi'_{3\beta})^\intercal=\Phi^\intercal_{3\alpha}$, while $\Phi'^*_{3\xi}=(\Phi'_{3\xi})^\intercal=\Phi^\intercal_{3\xi}$ for all $\xi\not\in\{\alpha,\beta\}$; $\Phi'^*_{4\alpha}=(\Phi'_{2\alpha})^\intercal=\Phi^\intercal_{4\beta}$, while $\Phi'^*_{4\xi}=(\Phi'_{2\xi})^\intercal=\Phi^\intercal_{2\xi}$ for all $\xi\ne\alpha$.

And from these equalities, according to Subsection~\ref{subsect:induction_Phi}, for the array $\Phict^{\prime * \prime}=\Pi^*\Phict'^*=\Pi_{\beta\alpha}^{\mathrm{rb}}\Phict'^*$ we obtain the following: $\Phi^{\prime * \prime}_{1\xi}=\Phi'^*_{1\xi}=\Phi^\intercal_{1\xi}$ for all $\xi\in\A$; $\Phi^{\prime * \prime}_{2\beta}=\Phi'^*_{4\alpha}=\Phi^\intercal_{4\beta}$, while $\Phi^{\prime * \prime}_{2\xi}=\Phi'^*_{2\xi}=\Phi^\intercal_{4\xi}$ for all $\xi\ne\beta$ (thus finally we have $\Phi^{\prime * \prime}_{2\xi}=\Phi^\intercal_{4\xi}$ for all $\xi\in\A$); $\Phi^{\prime * \prime}_{3\alpha}=\Phi'^*_{3\beta}=\Phi^\intercal_{3\alpha}$, while $\Phi^{\prime * \prime}_{3\xi}=\Phi'^*_{3\xi}=\Phi^\intercal_{3\xi}$ for all $\xi\not\in\{\alpha,\beta\}$ (thus finally we have $\Phi^{\prime * \prime}_{3\xi}=\Phi^\intercal_{3\xi}$ for all $\xi\ne\beta$); $\Phi^{\prime * \prime}_{4\xi}=\Phi'^*_{4\xi}=\Phi^\intercal_{2\xi}$ for all $\xi\ne\alpha$.

On the other hand, according to Subsection~\ref{subsect:dual_surface}, for the array $\Phict^*=\Inv\Phict$ we have the equalities $\Phi^*_{1\xi}=\Phi^\intercal_{1\xi}$, $\Phi^*_{2\xi}=\Phi^\intercal_{4\xi}$, $\Phi^*_{3\xi}=\Phi^\intercal_{3\xi}$, $\Phi^*_{4\xi}=\Phi^\intercal_{2\xi}$ for all $\xi\in\A$.

Comparing the obtained formulas for the elements of the arrays $\Phict^{\prime * \prime}$ and $\Phict^*$, we see that they coincide except for the functions with subscripts $3\beta$ and $4\alpha$. But, as we already noted above, according to Rule~3) from Subsection~\ref{subsect:Phi}, any three nodal functions from one rectangle uniquely determine the fourth function; in the present case, we obtain $\Phi^{\prime * \prime}_{3\beta}=\Phi^{\prime * \prime}_{4\beta}\circ(\Phi^{\prime * \prime}_{1\beta})^{-1}\circ\Phi^{\prime * \prime}_{2\beta}=\Phi^*_{4\beta}\circ(\Phi^*_{1\beta})^{-1}\circ\Phi^*_{2\beta}=\Phi^*_{3\beta}$, and similarly $\Phi^{\prime * \prime}_{4\alpha}=\Phi^*_{4\alpha}$. Thus, we have shown the identical equality of the arrays $\Phict^{\prime * \prime}=\Phict^*$, which proves the theorem for the elementary induction step $\Pi=\Pi_{\alpha\beta}^{\mathrm{rb}}$. For the remaining three steps the proof is completely similar (obviously, one need to take into account the assumption that the steps are not prohibited).

The theorem is proved.

\section{Applicability of the proposed approach to the study of circle diffeomorphisms with breaks}
\label{sect:verification}

In this section we will present arguments supporting our claim that the proposed approach is indeed applicable to the study of circle diffeomorphisms with breaks rather then being a merely speculative construction.

\subsection{Verification by the case $N=1$.}

Here we will show that the duality, which was successfully applied to prove rigidity theorems for circle diffeomorphisms with a single break in~\cite{Teplinsky08,KhaninTeplinsky13}, is the same duality as the one defined by us here in Section~\ref{sect:duality}, up to certain simple changes of coordinates.

For this purpose, let us rewrite in our current notation the commuting collection appearing in the works mentioned above (in the present case we are dealing with a \emph{commuting pair}, since there are only two functions). On the two-element alphabet $\A=\{\alpha,\omega\}$, consider the rotational scheme of special type
$$
\sigma=\{(\alpha\be,\omega\be,\alpha\en,\omega\en)\}=\left\{\left[\begin{array}{cc}
\alpha & \omega\\
\omega & \alpha \end{array}\right]\right\}, 
$$
which consists of a single cycle and corresponds to exchanging two adjacent intervals. The fixed element $\omega\be$ corresponds to the break point $x_{\omega\be}=0$ with break size $\nu=c^2$, $c\ne1$. The linear-fractional functions acting on the corresponding intervals, considered as elements of $\psl$, are given by
\begin{equation}\label{eq:N=1}
F_\alpha=\left(\begin{array}{cc}
c & a\\
-v & 1 \end{array}\right),\quad 
F_\omega=\left(\begin{array}{cc}
a & -ac\\
v+1-c & ac \end{array}\right),
\end{equation}
with certain restrictions on the parameters $a,v\in\R$ (in~\cite{Teplinsky08,KhaninTeplinsky13} these functions were denoted by $F_{a,v,c}$ and $G_{a,v,c}$). It is easy to verify that they satisfy a commutation relation of the form~(\ref{eq:commutation}), which in this case reduces to
\begin{equation}\label{eq:N=1:commuting}
F_\alpha\circ F_\omega\circ c^2\Id=F_\omega\circ F_\alpha.
\end{equation}

Following our approach, let us calculate the commuting pair $F^*_\alpha,F^*_\omega\in\psl$ dual to the commuting pair $F_\alpha,F_\omega$. The scheme dual to $\sigma$ is the rotational scheme of special type
$$
\sigma^*=\{(\omega\be,\alpha\be,\omega\en,\alpha\en)\}=\left\{\left[\begin{array}{cc}
\omega & \alpha\\
\alpha & \omega \end{array}\right]\right\}. 
$$ 
According to the algorithm described in~\ref{subsect:Phi}, we first successively obtain an array of eight nodal functions: $\Phi_{1\omega}=c^2\Id$, $\Phi_{2\alpha}=\Id$,  $\Phi_{1\alpha}=\Phi_{4\omega}=F_\omega\circ c^2\Id$, $\Phi_{2\omega}=\Phi_{3\alpha}=F_\alpha$,  $\Phi_{3\omega}=F_\omega\circ F_\alpha$, $\Phi_{4\alpha}=F_\alpha\circ F_\omega\circ c^2\Id$, and then derive from it the equalities $F^*_\alpha=(F_\omega^\intercal)^{-1}\circ c^{-2}\Id$, $F^*_\omega=F_\alpha^\intercal\circ c^{-2}\Id$. Using the formulas~(\ref{eq:N=1}), we obtain the dual (in our present sense) commuting pair explicitly:
$$
F^*_\alpha=\left(\begin{array}{cc}
a & c(c-v-1)\\
a & ac \end{array}\right),\quad 
F^*_\omega=\left(\begin{array}{cc}
c & -vc^2\\
a & c^2 \end{array}\right).
$$
These functions satisfy a commutation relation of the form~(\ref{eq:commutation}), which in this case is
$$
F^*_\omega\circ c^2\Id\circ F^*_\alpha=F^*_\alpha\circ F^*_\omega.
$$

Now let us compare the pair obtained according to our present approach with the duality described in~\cite{Teplinsky08,KhaninTeplinsky13}. There, the duality involution acted in the parameter space $a,v,c$ by the formula
\begin{equation}
\label{eq:N=1:Inv}
\Inv(a,v,c)=\left(\frac{c-v-1}{av},-\frac{v}{c},\frac{1}{c}\right),
\end{equation}
and hence, according to~(\ref{eq:N=1}), the commuting pair dual to the commuting pair $F_\alpha,F_\omega$ was considered to be
$$
\check F^*_\alpha=\left(\begin{array}{cc}
av & c(c-v-1)\\
av^2 & avc \end{array}\right),\quad 
\check F^*_\omega=\left(\begin{array}{cc}
c & -1\\
av & 1 \end{array}\right)
$$
(we recall that all matrices written here are regarded as elements of $\psl$, and hence we may multiply all their entries by any nonzero number, eliminating fractions in the present case). It satisfied the commutation relation corresponding to~(\ref{eq:N=1:commuting})
$$
\check F^*_\alpha\circ \check F^*_\omega\circ c^{-2}\Id=\check F^*_\omega\circ \check F^*_\alpha.
$$

It is easy to see that for the two variants of the commuting pair dual to $F_\alpha,F_\omega$, namely $\check F^*_\alpha,\check F^*_\omega$ (according to the construction of~\cite{Teplinsky08,KhaninTeplinsky13}) and $F^*_\alpha,F^*_\omega$ (according to our present construction), the following equalities hold:
$$
\check F^*_\alpha=v^{-1}\Id\circ F^*_\alpha\circ v\Id,\quad 
\check F^*_\omega=v^{-1}\Id\circ F^*_\omega\circ c^2v\Id.
$$
In these identities, multiplication on the left and on the right by $v^{-1}\Id$ and $v\Id$, respectively, is nothing else but a linear change of coordinates on $\R$, that is, a \emph{renormalization} of the given commuting pair (in~\cite{Teplinsky08,KhaninTeplinsky13}, successive  renormalizations of commuting pairs are considered, i.e., in addition to the induction process, the coordinates on $\R$ are intentionally changed linearly each time so that the initial interval $I_{\alpha\be}$ is always $[-1,0)$). Multiplication on the right by $c^2\Id$ transfers (in our present terminology) the role of the fixed element from $\omega\be$ to $\alpha\be$, so that instead of the equality $F^*_\omega\circ c^2\Id\circ F^*_\alpha=F^*_\alpha\circ F^*_\omega$ we have $\check F^*_\omega\circ \check F^*_\alpha\circ c^2\Id=\check F^*_\alpha\circ \check F^*_\omega$, which is related to switching the break size from $c^2$ to $1/c^2$ under the action of the involution~(\ref{eq:N=1:Inv}) in the parameter space.

The comparison carried out above allows us to conclude that the duality defined by us in this paper for $N\ge1$, in the particular case of $N=1$ is effectively the same duality that we successfully used in~\cite{Teplinsky08,KhaninTeplinsky13} to reveal the hyperbolic structure in the space of commuting pairs under the action of the renormalization operator, and ultimately to prove important results concerning rigidity for circle diffeomorphisms with a single break. Thus, the construction we present here is indeed a generalization to arbitrary $N\ge1$ of the construction we developed previously for the case $N=1$, and which has already demonstrated its usefulness for research.

\subsection{Infinite induction for linear-fractional interval exchange\\ transformations}

Despite the example given in~\ref{subsect:counterexample}, we claim that for any circle diffeomorphism with breaks, which has an irrational rotation number and whose trajectories of all $N\ge1$ break points are disjoint, there exist infinite algorithms of its successive renormalization and corresponding induction algorithms for linear-fractional interval exchange transformations (to whose space these renormalizations converge in the limit under certain combinatorial conditions and sufficient smoothness), corresponding to commuting collections exactly of the form described in~\ref{subsect:commuting_collections}, using only elementary induction steps which are applicable in the sense of~\ref{subsect:induction} and allowed in the sense of~\ref{subsect:prohibited}.

Since every sufficiently smooth (in particular, $C^{2+\varepsilon}$-smooth, $\varepsilon>0$, as in~\cite{KhaninTeplinsky13}) circle diffeomorphism with $N\ge1$ breaks and an irrational rotation number is topologically conjugate to the linear rotation of the circle by this number (by the classical Denjoy theory), then, according to the linear case considered in~\cite{Teplinsky24}, for such a circle diffeomorphism with breaks the first return map of its trajectories to any given union of a finite number of circle arcs is a nonlinear rotational interval exchange transformation. Thus, it is only necessary to choose such a union of arcs and such a sequence of induction steps in a suitable way so that the induction algorithm defined by them produces, on the one hand, infinite sequences of nonlinear interval exchange transformations for the given circle diffeomorphism with breaks and, on the other hand, also infinite sequences of linear-fractional interval exchange transformations specified by the commuting collections described in~\ref{subsect:commuting_collections} for rotational schemes of special type. In our approach, the mentioned union should consist of $N$ arcs, each containing exactly one break point (which will be a fixed endpoint and will remain in place throughout the entire infinite induction/renormalization process), while all other endpoints will belong to the trajectories of these $N$ fixed points.

There is no great difficulty in constructing such algorithms as described above; in particular, they can be constructed on the basis of successive \emph{dynamical partitions} of the circle into \emph{fundamental segments}, which were considered as early as in Herman's foundational work~\cite{Herman79} and which we used in~\cite{KhaninTeplinsky13}. The most obvious of the possible algorithms are based on \emph{rotational schemes in canonical form}, which were defined in Section~5.7 of~\cite{Teplinsky24}.

However, we must note that the description of infinite induction algorithms requires a considerable amount of text and logically falls beyond the scope of the present paper. We plan to describe such algorithms in detail in our subsequent works, as well as to extend the theory developed here to accelerated induction of Zorich type~\cite{Zorich96}, which is commonly used to construct a sequence of renormalizations of circle diffeomorphisms with singularities (in particular, by us in~\cite{KhaninTeplinsky13}).

Finally, let us point out that the mere existence of infinite induction for linear-fractional interval exchange transformations given by commuting collections does not imply that the dual commuting collections will produce nonlinear interval exchange transformations throughout the entire infinite process. In the case $N=1$, this was ensured by geometric bounds established for the regions of parameter space within which the renormalization algorithm acted~\cite{Teplinsky08}. In the case $N>1$, this problem remains open and will require a separate investigation.

\section{Summary}
\label{sect:summary}

In this paper, within the framework of the concept of interval rearrangement ensembles (IRE) proposed by us in~\cite{Teplinsky23}, and in view of the understanding of rotational interval exchange transformations obtained in~\cite{Teplinsky24}, we defined the notion of a commuting collection $\Fect$ of $2N$, $N\ge1$, increasing linear-fractional functions, which is a generalization of the previously known notion of a commuting pair (the case $N=1$), which serves, in particular, as a tool in the study of renormalizations of circle diffeomorphisms with a single break. Under certain broad enough assumptions, the space of such commuting collections is approached by successive renormalizations of circle diffeomorphisms with $N$ breaks, similarly to how renormalizations of circle diffeomorphisms with a single break approach the space of commuting pairs.

In our definition of a commuting collection, which functions are regarded as elements of the projective special linear group $\psl$, the commutation relations~(\ref{eq:commutation}) include in themselves the full information about the rotational interval exchange scheme of special type $\sigma$ and the break sizes $\nu_i>0$, $1\le i\le N$.

For commuting collections, we defined an involutive duality operation $\Inv$, which associates with a collection $\Fect$ its dual collection $\Fect^*$ with the rotational scheme of special type $\sigma^*$ dual to $\sigma$ (in the sense of duality for IRE schemes) and with the same break sizes $\nu_i$, $1\le i\le N$.

It was shown (Theorem~1) that the defined duality operation reverses time for the induction process, i.e., it conjugates each elementary induction step to a certain reverse induction step.

This result, apart from being interesting by itself, also opens a potential possibility of proving hyperbolicity of the renormalization operator for circle diffeomorphisms with multiple breaks---and thereby moving closer to a general result concerning the rigidity of such maps, similar to the result on the rigidity of circle diffeomorphisms with a single break obtained in~\cite{KhaninTeplinsky13}.

As a separate point, we demonstrated that the duality that played a key role in proving the above-mentioned result of~\cite{KhaninTeplinsky13} coincides with the duality defined by us in this paper (for the case $N=1$) up to a change of coordinates, which confirms the relevance of the developed theory.

In our further works, we plan to describe infinite induction algorithms for commuting collections, which directly correspond to renormalization algorithms for circle diffeomorphisms with multiple breaks.


\begin{thebibliography}{99}

\bibitem{KhaninVul91} K.~M. Khanin, E.~B. Vul. 
Circle homeomorphisms with weak discontinuities. 
In {\em  Adv.\ Soviet Math., 3. Dynamical systems and statistical mechanics} 
Amer.\ Math.\ Soc., Providence, RI, 57--98 (1991).

\bibitem{Yampolsky03} M.~Yampolsky. 
Global renormalization horseshoe for critical circle maps. {\em Commun.\ Math.\ Phys.}, {\bf 240} (2003), 75--96.

\bibitem{KhaninTeplinsky13} K.~Khanin, A.~Teplinsky.
Renormalization horseshoe and rigidity for circle diffeomorphisms with breaks.
{\em Commun.\ Math.\ Phys.} {\bf 320} (2013), 347--377.

\bibitem{Teplinsky08} 
O.~Y.~Teplins’kyi.
Hyperbolic horseshoe for circle diffeomorphisms with break.
{\em Nonlinear Oscill.} 
{\bf 11} (2008), 114--134.

\bibitem{KhaninTeplinsky07} K.~Khanin, A.~Teplinsky. 
Robust rigidity for circle diffeomorphisms with singularities.
{\em Invent.\ Math.} {\bf 169} (2007), 193--218.

\bibitem{Zorich96} A.~Zorich. 
Finite Gauss measure on the space of interval exchange transformations. Lyapunov exponents. {\em Annales de l'Institut Fourier.} {\bf 46} (1996), 325--370.

\bibitem{Veech78} W.~A.~Veech. 
Interval exchange transformations. 
{\em J.\ Analyse Math.} {\bf 33} (1978), 222--272.

\bibitem{Rauzy79} G.~Rauzy. 
\' Echanges d’intervalles et transformations induites. 
{\em Acta Arith.} {\bf 34} (1979), 315--328.

\bibitem{Keane75} M.~Keane. 
Interval exchange transformations. 
{\em Math.\ Z.} {\bf 141} (1975), 25--31.

\bibitem{Teplinsky23} 
A.~Teplinsky.
Interval Rearrangement Ensembles.
{\em Ukr.\ Math.\ J.} 
{\bf 75} (2023), 282--304.

\bibitem{Teplinsky24} 
A.~Teplinsky.
Rotational Interval Exchange Transformations.
{\em Ukr.\ Math.\ J.} 
{\bf 76} (2024), 501--521.

\bibitem{CunhaSmania13} K.~Cunha, D.~Smania. 
Renormalization for piecewise smooth homeomorphisms on the circle. 
{\em Ann.\ Inst.\ H.~Poincar\'e Anal.\ Non Lin\'eaire} {\bf 30} (2013), 441--462.

\bibitem{Veech82} W.~A.~Veech. 
Gauss measures for transformations on the space of interval exchange maps. 
{\em Ann.\ of Math.} {\bf 115} (1982), 201--242.

\bibitem{Herman79} M.-R.~Herman. 
Sur la conjugaison differentiable des diffeomorphismes du cercle a des
rotations. 
{\em I.~H.~E.~S.\ Publ.\ Math.} {\bf 49} (1979), 5--233.

\end{thebibliography}
\end{document}